\documentclass[matprg]{svjour3}

\usepackage{amsfonts}
\usepackage{latexsym}
\usepackage{amsmath, amssymb}

\usepackage{algorithm}\usepackage{algpseudocode}

\newcommand{\beq}{\begin{equation}}
\newcommand{\eeq}{\end{equation}}
\newcommand{\beqa}{\begin{eqnarray}}
\newcommand{\eeqa}{\end{eqnarray}}
\newcommand{\beqas}{\begin{eqnarray*}}
\newcommand{\eeqas}{\end{eqnarray*}}
\newcommand{\bi}{\begin{itemize}}
\newcommand{\ei}{\end{itemize}}
\newcommand{\ba}{\begin{array}}
\newcommand{\ea}{\end{array}}

\newcommand{\nn}{\nonumber}

\def\eqnok#1{(\ref{#1})}
\def\argmin{{\rm argmin}}

\def\Argmin{{\rm Argmin}}

\def\exp{{\rm exp}}

\def\vgap{\vspace*{.1in}}

\def\setU{{X}}

\def\R{{\Re}}

\def\SO{{\rm SO}}

\def\ProxS{{\rm PS}}
\def\SProxS{{\rm SPS}}
\def\tx{{\tilde x}}
\def\ty{{\tilde y}}
\def\hu{{u^*}}

\def\cX{{\cal X}}

\def\ux{{\underline x}}
\def\bx{{\bar x}}

\newcommand{\bbe}{\mathbb{E}}
\def\prob{\mathop{\rm Prob}}

\newcommand{\bbr}{\mathbb{R}}

\def\w{\omega}

\title{
Gradient Sliding for Composite Optimization
\thanks{The author of this paper was partially supported by
    NSF grant CMMI-1000347, DMS-1319050, 
    ONR grant N00014-13-1-0036 and
    NSF CAREER Award CMMI-1254446.}}
\author{
    Guanghui Lan
    \thanks{Department of Industrial and Systems
    Engineering, University of Florida, Gainesville, FL, 32611.
    (email: {\tt glan@ise.ufl.edu}).}
}

\begin{document}

\maketitle

\begin{abstract}
We consider in this paper a class of composite optimization problems whose objective function
is given by the summation of a general smooth and nonsmooth component, together with
a relatively simple nonsmooth term. We present a new class of first-order
methods, namely the gradient sliding algorithms, which can skip the computation of the
gradient for the smooth component from time to time. As a consequence, 
these algorithms require only ${\cal O}(1/\sqrt{\epsilon})$ gradient evaluations for
the smooth component in order to find an $\epsilon$-solution for the composite problem, while
still maintaining the
optimal ${\cal O}(1/\epsilon^2)$ bound on the total number of subgradient 
evaluations for the nonsmooth component.
We then present a stochastic counterpart for these algorithms and establish 
similar complexity bounds for solving an important class of stochastic composite optimization problems. 
Moreover, if the smooth component in the composite function is strongly convex, 
the developed gradient sliding algorithms can significantly reduce
the number of graduate and subgradient evaluations for the smooth and nonsmooth component
to ${\cal O} (\log (1/\epsilon))$ and ${\cal O}(1/\epsilon)$,
respectively. 
Finally, we generalize these algorithms to the case when the smooth component is replaced by a nonsmooth
one possessing a certain bi-linear saddle point structure.

\vspace{.1in}


\vspace{.07in}


\end{abstract}

\vspace{0.1cm}

\setcounter{equation}{0}
\section{Introduction}
In this paper, we consider a class of composite convex programming (CP) problems given in the form of
\beq \label{cp}
\Psi^* \equiv \min_{x \in X} \left\{ \Psi(x) := f(x) + h(x) + \cX(x) \right\}.
\eeq
Here, $X \subseteq \bbr^n$ is a closed convex set,  $\cX$ is relatively simple convex function, and $f: X \to \bbr$ 
and $h: X \to \bbr$, respectively, are general smooth and
nonsmooth convex functions satisfying
\begin{align}
f(x) & \le f(y) + \langle \nabla f(y), x - y \rangle + \frac{L}{2} \|x - y\|^2, \,  \forall x, y \in X, \label{smoothf1}\\
h(x) &\le h(y) + \langle h'(y), x - y \rangle + M \|x - y\|, \,  \forall x, y \in X, \label{nonsmoothh1}
\end{align}
for some $L > 0$ and $M > 0$, where $h'(x) \in \partial h(x)$. 
Composite problem of this type
appears in many data analysis applications, where either $f$ or $h$
corresponds to a certain data fidelity term, while the other components in $\Psi$ denote regularization
terms used to enforce certain structural properties for the obtained solutions.

Throughout this paper, we assume that
one can access the first-order information of $f$ and $h$ separately. More specifically,
in
the deterministic setting, we can compute the
exact gradient $\nabla f(x)$ and a subgradient $h'(x) \in \partial h(x)$ for any $x \in X$.
We also consider the stochastic situation where only a
stochastic subgradient of the nonsmooth component $h$ is
available.
The main goal of this paper to provide a better theoretical understanding on how many number of 
gradient evaluations of $\nabla f$ and subgradient evaluations
of $h'$ are needed in order to find a certain approximate solution of \eqnok{cp}.

Most existing
first-order methods for solving \eqnok{cp} require the computation
of both $\nabla f$ and $h'$ in each iteration. In particular,
since the objective function $\Psi$ in \eqnok{cp} is nonsmooth, these
algorithms would require ${\cal O}(1/\epsilon^2)$ first-order iterations, and hence ${\cal O}(1/\epsilon^2)$
evaluations
for both $\nabla f$ and $h'$ to find an $\epsilon$-solution
of \eqnok{cp}, i.e., a point $\bar x \in X$ s.t. $\Psi(\bar x) - \Psi^* \le \epsilon$. Much
recent research effort has been directed to reducing the impact of the Lipschitz constant
$L$ on the aforementioned complexity bounds for composite optimization.  
For example, Juditsky, Nemirovski and Travel showed in \cite{jnt08} that
by using a variant of the mirror-prox method,
the number of evaluations for $\nabla f$ and $h'$ required
to find an $\epsilon$-solution of \eqnok{cp} can be bounded
by
\[
{\cal O} \left(\frac{L_f}{\epsilon} + \frac{M^2}{\epsilon^2} \right).
\]
By developing an enhanced version of
Nesterov's accelerated gradient method~\cite{Nest83-1,Nest04}, Lan~\cite{Lan10-3} further showed that
the above bound can be improved to 
\beq \label{sub_bnd}
{\cal O} \left(\sqrt{\frac{L_f}{\epsilon}} + \frac{M^2}{\epsilon^2} \right).
\eeq
It is also shown in \cite{Lan10-3} that similar bounds hold for the stochastic case
where only unbiased estimators for $\nabla f$ and $h'$ are available.
It is observed in \cite{Lan10-3} that such a complexity bound is not improvable if 
one can only access the first-order information for the summation of $f$ and $h$ all together.

Note, however, that it is unclear whether the complexity bound in \eqnok{sub_bnd} is optimal if
one does have access to the
first-order information of $f$ and $h$ separately. In particular, one would  expect that the number of evaluations 
for $\nabla f$ can be bounded by
${\cal O}(1/\sqrt{\epsilon})$, 
if the nonsmooth term $h$ in \eqnok{cp}
does not appear (see \cite{Nest07-1,tseng08-1,BecTeb09-2}). However, it is unclear whether such a bound 
still holds for the more general composite problem in \eqnok{cp} without significantly increasing the bound
in \eqnok{sub_bnd} on the number of subgradient
evaluations for $h'$.
It should be pointed out that in many applications the bottleneck of first-order methods exist in the 
computation of $\nabla f$ rather than that of $h'$. To motivate our study, let us mention a few such examples.
\begin{itemize}
\item [a)] In many inverse problems, we need to enforce certain block sparsity (e.g., total variation and
overlapped group Lasso) by solving the problem of 
$
\min_{x \in \bbr^n} \|Ax - b\|_2^2 + r(B x).
$
Here $A: \bbr^n \to \bbr^m$ is a given linear operator, $b \in \bbr^m$ denotes the collected observations,
$r: \bbr^p \to \bbr$ is a relatively simple nonsmooth convex function (e.g., $r = \|\cdot\|_1$), and $B:\bbr^n \to \bbr^p$
is a very sparse matrix. In this case, evaluating the gradient of $\|A x - b\|^2$ requires
${\cal O}(m n)$ arithmetic operations, while the computation of $r'(B x)$ 
only needs ${\cal O}(n + p)$ arithmetic operations.
\item [b)] In many machine learning problems, we need to minimize a regularized
loss function given by
$
\min_{x \in \bbr^n} \bbe_\xi [l(x, \xi)] + q(Bx).
$
Here $l: \bbr^n \times \bbr^d \to \bbr$ denotes a certain simple loss function, $\xi$ is
a random variable with unknown distribution, $q$ is a certain smooth convex function,
and $B:\bbr^n \to \bbr^p$ is a given linear operator. In this case, the computation of the stochastic subgradient
for the loss function $ \bbe_\xi [l(x, \xi)]$ requires only ${\cal O}(n+d)$ arithmetic operations,
while evaluating the gradient of $q(B x)$ needs ${\cal O}(n p)$ arithmetic operations.
\item [c)] In some cases, the computation of $\nabla f$ involves a black-box simulation procedure,
the solution of an optimization problem, or a partial differential equation, 
while the computation of $h'$ is given explicitly.
\end{itemize}
In all these cases mentioned above, it is desirable to reduce the number of gradient evaluations
of $\nabla f$ to improve the overall efficiency for solving the composite problem \eqnok{cp}.

Our contribution can be briefly summarized as follows.
Firstly, we present a new class of first-order methods, namely the gradient sliding
algorithms, and show that the number of gradient evaluations for $\nabla f$ 
required by these algorithms to find an $\epsilon$-solution of \eqnok{cp} can
be significantly reduced from \eqnok{sub_bnd} to 
\beq \label{bnd_grad}
{\cal O} \left( \sqrt{\frac{L}{\epsilon}}\right),
\eeq
while the total number of subgradient evaluations for $h'$ is still bounded
by \eqnok{sub_bnd}. The basic scheme of these algorithms is to skip the computation
of $\nabla f$ from time to time so that only ${\cal O}(1/\sqrt{\epsilon})$ gradient
evaluations are needed in the ${\cal O}(1/\epsilon^2)$ iterations required
to solve \eqnok{cp}. Such an algorithmic framework originated from the simple idea of incorporating
an iterative procedure to solve the subproblems in the aforementioned accelerated proximal gradient
methods, although the analysis of these gradient sliding algorithms appears to be more technical and involved.


Secondly, we consider the stochastic case where the nonsmooth term $h$ is represented by
a stochastic oracle ($\SO$), which,
for a given search point $u_t \in X$, outputs a vector $H(u_t, \xi_t)$ such that (s.t.)
\begin{align}
\bbe[H(u_t, \xi_t)] = h'(u_t) \in \partial h(u_t), \label{assum_FO_a} \\
\bbe[\|H(u_t, \xi_t) - h'(u_t) \|_*^2] \le \sigma^2, \label{assum_FO_b}
\end{align}
where $\xi_t$ is a random vector independent of the search points $u_t$.
Note that $H(u_t, \xi_t)$ is referred to as a stochastic subgradient of
$h$ at $u_t$ and its computation is often much cheaper than the exact subgradient
$h'$. Based on the gradient sliding techniques, we
develop a new class of stochastic approximation type algorithms and show that
 the total number gradient evaluations
of $\nabla f$ required by these algorithms to find a stochastic $\epsilon$-solution of \eqnok{cp},
i.e., a point $\bar x \in X$ s.t. $\bbe[\Psi(\bar x) - \Psi^*] \le \epsilon$, can 
still be bounded by \eqnok{bnd_grad}, while the total number of 
stochastic subgradient evaluations can be bounded by
\[
{\cal O} \left( \sqrt{\frac{L}{\epsilon}} + \frac{M^2 + \sigma^2}{\epsilon^2}\right).
\]
We also establish large-deviation results associated with these 
complexity bounds under certain ``light-tail'' assumptions on the stochastic
subgradients returned by the $\SO$.

Thirdly, we generalize the gradient sliding algorithms for solving two important 
classes of composite problems given in the form of \eqnok{cp}, but with $f$ satisfying
additional or alterative assumptions. We first assume that $f$ is not only smooth,
but also strongly convex, and show that the number of evaluations for
$\nabla f$ and $h'$ can be significantly reduced 
from ${\cal O}(1/\sqrt{\epsilon})$ and ${\cal O}(1/\epsilon^2)$, respectively,
to ${\cal O} (\log (1/\epsilon))$ and ${\cal O}(1/\epsilon)$. 
We then consider the case when $f$ is nonsmooth, but can be closely approximated by
a class of smooth functions. By incorporating a novel smoothing scheme due to Nesterov~\cite{Nest05-1}
into the gradient sliding algorithms, we show that the number of gradient evaluations
can be bounded by ${\cal O}(1/\epsilon)$, while the optimal ${\cal O}(1/\epsilon^2)$ bound
on the number of subgradient evaluations of $h'$ is still retained.
 
This paper is organized as follows. In Section~\ref{sec_prelim}, we provide some preliminaries on the
prox-functions and a brief review on
existing proximal gradient methods for solving \eqnok{cp}. In Section~\ref{sec_DGS}, we present the gradient sliding
algorithms and establish their convergence properties for solving problem \eqnok{cp}.
Section~\ref{sec_SGS} is devoted to stochastic gradient sliding algorithms for solving a class of stochastic
composite problems. In Section~\ref{sec_gen}, we generalize the gradient sliding algorithms for the situation 
where $f$ is smooth and strongly convex, and for the case when $f$ is nonsmooth but can be closely approximated
by a class of smooth functions. Finally, some concluding remarks are made in Section~\ref{sec_remark}.

\vgap

\noindent {\bf Notation and terminology.}
We use $\|\cdot\|$ to denote an arbitrary norm in $\bbr^n$,
which is not necessarily associated the inner product $\langle \cdot, \cdot \rangle$.
We also use $\|\cdot\|_*$ to denote the conjugate of $\|\cdot\|$.
For any $p \ge 1$, $\|\cdot\|_p$ denotes the standard $p$-norm in  $\bbr^n$, i.e.,
\[
\|x\|_p^p =  \sum_{i=1}^n |x_i|^p, \qquad \mbox{for any } x \in \bbr^n.
\]
For any convex function $h$, $\partial h(x)$ is the set of subdifferential at $x$.
Given any $X \subseteq \bbr^n$, we say a convex function $h:X \to \bbr$ is nonsmooth
if $|h(x) - h(y)| \le M_h \|x - y\|$ for any $x, y \in X$. In this case, it can be shown that
\eqnok{nonsmoothh1} holds with $M = 2 M_h$ (see Lemma 2 of \cite{Lan10-3}).
We say that a convex function $f: X \to \bbr$ is smooth if it
is Lipschitz continuously differentiable with Lipschitz constant $L>0$, i.e.,
$
\|\nabla f(y) - \nabla f(x)\|_* \le L\|y-x\|$
for any $x, y \in X$,
which clearly implies \eqnok{smoothf1}.

For any real number $r$, $\lceil r \rceil$ and $\lfloor r \rfloor$ denote the nearest integer to
$r$ from above and below, respectively. $\bbr_+$ and $\bbr_{++}$, respectively, denote the set of nonnegative 
and positive real numbers. ${\cal N}$ denotes the set of natural numbers $\{1, 2, \ldots\}$.

\setcounter{equation}{0}
\section{Review of the proximal gradient methods}
In this section, we provide a brief review on the proximal gradient methods from
which the proposed gradient sliding algorithms originate, and point out a few problems associated with
these existing algorithms when applied to solve problem \eqnok{cp}.

\subsection{Preliminary: distance generating function and prox-function}\label{sec_prelim}
In this subsection, we review the concept of prox-function
(i.e., proximity control function), which plays an important role in the recent development
of first-order methods for convex programming. The goal of using the prox-function
in place of the usual Euclidean distance is to allow the developed algorithms to
get adapted to the geometry of the feasible sets.

We say that a function $\w:\,X\to \bbr$ is a {\em
distance generating function} with modulus $\nu>0$
with respect to $\|\cdot\|$,
if $\w$ is continuously differentiable and strongly convex with parameter
$\nu$ with respect to  $\|\cdot\|$, i.e.,
\begin{equation}\label{s27}
 \langle x-z, \nabla \w(x)-\nabla\w(z) \rangle \geq\nu
\|x-z\|^2,\;\;\forall x,z\in X.\end{equation}
\par
The {\em prox-function} associated with $\w$ is given by
\begin{equation}\label{s450}
V(x,z) \equiv V_\w(x,z)=\w(z)-[\w(x)+\langle \nabla \w(x), z-x \rangle].
\end{equation}
The prox-function $V(\cdot,\cdot)$ is also called the Bregman's distance,
which was initially studied by Bregman \cite{Breg67} and later by many others
(see \cite{AuTe06-1,BBC03-1,Kiw97-1} and references therein).
In this paper, we assume that the prox-function $V(x,z)$ is chosen such
that the solution of
\beq \label{prox_mapping}
\arg\min\limits_{u \in X} \left\{ \langle g, u \rangle
+ V(x, u) + \cX(u) \right\}
\eeq
is easily computable for any $g \in {\cal E}^*$ and $x \in X$.
Some examples of these prox-functions are given in \cite{GhaLan12-2a}.

If there exists a constant ${\cal Q}$ such that $V(x, z) \le {\cal Q} \|x - z\|^2 / 2$
for any $x, z\in X$, then we say that the prox-function $V(\cdot, \cdot)$ is growing
quadratically. Moreover, the smallest constant ${\cal Q}$ satisfying the previous
relation is called the {\sl quadratic growth constant} of $V(\cdot, \cdot)$. 
Without loss of generality, we assume that ${\cal Q} = 1$ for the
prox-function $V(x,z)$ if it grows quadratically, i.e.,
\beq \label{quad_grow}
V(x, z) \le \frac{1}{2} \|x - z\|^2, \ \ \ \forall x, z \in X.
\eeq
Indeed, if ${\cal Q} \neq 1$, we can multiply the
corresponding distance generating function $\w$ by $1/{\cal Q}$ and the
resulting prox-function will satisfy \eqnok{quad_grow}.

\subsection{Proximal gradient methods}
In this subsection, we briefly review a few possible first-order methods
for solving problem \eqnok{cp}. 

We start with the simplest proximal 
gradient method which works for the case when the nonsmooth component $h$ 
does not appear or is relatively simple
(e.g., $h$ is affine). For a given $x \in X$, let
\beq \label{def_big_phik}
m_\Psi(x, u) := l_f(x, u) + h(u) + \cX(u), \ \ \forall u \in X,
\eeq
where
\beq \label{def_lf}
l_f(x;y) := f(x) + \langle\nabla f(x), y - x \rangle.
\eeq
Clearly, by the convexity of $f$ and \eqnok{smoothf1}, we
have
\begin{align*}
m_\Psi(x, u) &\le \Psi(u) \le m_\Psi(x, u) + \frac{L}{2} \|u - x\|^2 
\le m_\Psi(x, u) +  \frac{L}{2 \nu} V(x, u) 
\end{align*}
for any $u \in X$, where the last inequality follows 
from the strong convexity of $\w$. Hence, 
$m_\Psi(x,u)$ is a good approximation of $\Psi(u)$ when $u$ is ``close"
enough to $x$. In view of this observation, we update the search point
$x_{k} \in X$ at the $k$-th iteration of the proximal gradient method by
\beq \label{define-prox-mapping}
x_{k} = \argmin_{u \in X} \left\{l_f(x_{k-1}, u) + h(u) + \cX(u) + \beta_k V(x_{k-1}, u) \right\},
\eeq
Here, $\beta_k > 0$ is a parameter which determines how well we ``trust" the 
proximity between $m_\Psi(x_{k-1}, u)$ and $\Psi(u)$. In particular, a larger value of
$\beta_k$ implies less confidence on $m_\Psi(x_{k-1},u)$ and results in a smaller step 
moving from $x_{k-1}$ to $x_{k}$. It is well-known that the number of iterations
required by the proximal gradient method for finding an $\epsilon$-solution
of \eqnok{cp} can be bounded by ${\cal O}(1/\epsilon)$.

The efficiency of the above proximal gradient method can be significantly improved
by incorporating a multi-step acceleration scheme.
The basic idea of this scheme is to introduce
three closely related search sequences, namely, $\{\ux_k\}$, $\{x_k\}$, and $\{\bx_k\}$,
which will be used to build the model $m_\Psi$, control the proximity
between $m_\Psi$ and $\Psi$, and compute the output solution, respectively. More specifically,
these three sequences are updated according to
\begin{align}
\ux_k &= (1 - \gamma_k)  \bx_{k-1} + \gamma_k x_{k-1}, \label{acc1} \\
x_k & = \argmin_{u \in X} \left\{\Phi_k(u) :=  l_f(\ux_k, u) + h(u) + \cX(u) + \beta_k V(x_{k-1}, u) \right\}, \label{acc2}\\
\bx_k &= (1-\gamma_k) \bx_{k-1} + \gamma_k x_k, \label{acc3}
\end{align}
where  $\beta_k \ge 0$ and $\gamma_k \in [0,1]$ are given parameters for the algorithm.
Clearly, \eqnok{acc1}-\eqnok{acc3} reduces to \eqnok{define-prox-mapping},
if $\bx_0 = x_0$ and $\gamma_k$ is set to a constant, i.e., $\gamma_k =\gamma$ for some $\gamma \in [0,1]$
for all $k \ge 1$. However, by properly specifying $\beta_k$ and $\gamma_k$, e.g.,
$\beta_k = 2 L / k$ and $\gamma_k = 2/(k+2)$, one can show that the above
accelerated proximal gradient method can find an $\epsilon$-solution of \eqnok{cp}
in at most ${\cal O}(1/\sqrt{\epsilon})$ iterations. Since each iteration of
this algorithm requires only one evaluation of $\nabla f$, the total number of gradient evaluations
of $\nabla f$ can also be bounded by ${\cal O}(1/\sqrt{\epsilon})$.

One crucial problem associated with the aforementioned proximal gradient type methods
is that the subproblems \eqnok{define-prox-mapping} and \eqnok{acc2}
are difficult to solve when $h$ is a general nonsmooth convex function. 
To address this issue, one can possibly apply an enhanced 
accelerated gradient method by Lan~\cite{Lan10-3} (see also \cite{GhaLan12-2a,GhaLan13-1}).
This algorithm is obtained by replacing 
$h(u)$ in \eqnok{acc2} with
\beq \label{def_lh}
l_h(\ux_k;u) := h(\ux_k) + \langle h'(\ux_k), u - \ux_k \rangle
\eeq
for some $h'(\ux_k) \in \partial h(\ux_k)$.
As a result, the subproblems in this algorithm become easier to solve. Moreover, 
with a proper selection of $\{\beta_k\}$ and $\{\gamma_k\}$, 
this approach can find an $\epsilon$-solution of \eqnok{cp} 
in at most
\beq \label{iter_bnd}
{\cal O} \left\{\sqrt{\frac{L V(x_0, x^*)}{\epsilon}} + \frac{M^2 V(x_0, x^*) }{\epsilon^2} \right\}
\eeq
iterations. Since each iteration requires one computation of $\nabla f$ and $h'$,
the total number of evaluations for $f$ and $h'$ is bounded by ${\cal O}(1/\epsilon^2)$. 
As pointed out in \cite{Lan10-3}, this bound in \eqnok{iter_bnd} is not improvable if one can only 
compute the subgradient of the composite function $f(x) + h(x)$ as a whole. 
However, as noted in Section 1, we do have access to separate first-order information 
about $f$ and $h$
in many applications.
One interesting problem is whether we can further improve the performance of 
proximal gradient type methods in the latter case.

\setcounter{equation}{0}
\section{Deterministic gradient sliding} \label{sec_DGS}
Throughout this section, we consider the deterministic case where exact subgradients of $h$ are available.
By presenting a new class of proximal gradient methods, namely the gradient sliding (GS) method,
we show that one can significantly 
reduce the number of gradient evaluations for $\nabla f$ required to solve \eqnok{cp}, while
maintaining the optimal bound on the total number of subgradient evaluations for $h'$.

The basic idea of the GS method is to incorporate an iterative procedure to
approximately solve the subproblem \eqnok{acc2} 
in the accelerated proximal gradient methods.
A critical observation in our development of the GS method is that
one needs to compute a pair of closely related approximate solutions of
problem \eqnok{acc2}. One of them will be used in place of $x_k$ in \eqnok{acc1}
to construct the model $m_\Psi$, while the other one will be used in place of
$x_k$ in \eqnok{acc3} to compute the output solution $\bx_k$.
Moreover, we show that such a pair of approximation solutions can 
be obtained by applying a simple subgradient projection type subroutine.
We now formally describe this algorithm as follows.

\begin{algorithm} [H]
	\caption{The gradient sliding (GS) algorithm}
	\label{algGeneric}
	\begin{algorithmic}
\State 
\noindent {\bf Input:} Initial point $x_0 \in X$ and iteration limit $N$.

\State Let $\beta_k \in \bbr_{++}, \gamma_k \in \bbr_+$, and $T_k \in {\cal N}$, $k = 1, 2, \ldots$,
be given and
set $\bar x_0 = x_0$. 
 
\For {$k=1, 2, \ldots, N$ }
\State 1. Set $\underline x_k = (1 - \gamma_k) \bar x_{k-1} + \gamma_k x_{k-1}$, and let 
$g_k(\cdot) \equiv l_f(\underline x_{k-1}, \cdot)$ be defined in \eqnok{def_lf}.
\State 2. Set 
\beq \label{def_ut}
(x_k, \tilde x_k) = \ProxS(g_k, x_{k-1}, \beta_k, T_k);
\eeq
\State 3. Set $\bar x_k = (1-\gamma_k) \bar x_{k-1} + \gamma_k \tilde x_k$. 
\EndFor
\State 
\noindent {\bf Output:} $\bx_N$.

\Statex
\Statex The  $\ProxS$ (prox-sliding) procedure called at step 2 is stated as follows.\Procedure {$(x^+, \tilde x^+) = \ProxS$}{$g$, $x$, $\beta$, $T$}\State Let the parameters $p_t \in \bbr_{++}$ and $\theta_t \in [0,1]$, $t = 1, \ldots$, be given. Set 
 $u_{0} = \tilde u_0 = x$.
\State {\bf for} $t = 1, 2, \ldots, T$ {\bf do}
\begin{align}
u_{t} &= \argmin_{u \in X} \left\{ g(u) + l_h(u_{t-1},u) + \beta V(x, u) + \beta p_t V(u_{t-1}, u) + \cX(u) \right\},\label{def_ut_iter}\\
\tilde u_t &=  (1-\theta_t) \tilde u_{t-1} + \theta_t u_t. \label{def_but}
\end{align}
\State {\bf end for}
\State Set $x^+ = u_T$ and  $\tilde x^+ = \tilde u_T$.\EndProcedure
	\end{algorithmic}
\end{algorithm}

Observe that when supplied with an 
affine function $g(\cdot)$, prox-center $x \in X$, parameter $\beta$, and sliding period $T$, 
the $\ProxS$ procedure computes a
pair of approximate solutions $(x^+, \tilde x^+) \in X \times X$ for the
problem of:
\beq \label{generic_subproblem}
\argmin_{u \in X} \left\{\Phi(u) := g(u) + h(u) + \beta V(x, u) + \cX(u) \right\}.
\eeq
Clearly, problem \eqnok{generic_subproblem} is equivalent to \eqnok{acc2}
when the input parameters are set to \eqnok{def_ut}.
Since the same affine function $g(\cdot) = l_f(\ux_{k-1}, \cdot)$ has been used throughout
the $T$ iterations of the $\ProxS$ procedure, we skip 
the computation of the gradients of $f$ when performing the $T$ projection steps
in \eqnok{def_ut_iter}. This differs from the 
accelerated gradient method in~\cite{Lan10-3}, where one needs to
compute $\nabla f + h'$ in each projection step.

\vgap

A few more remarks about the above GS algorithm are in order.
Firstly, we say that an outer iteration of the GS algorithm occurs whenever $k$ in Algorithm~\ref{algGeneric}
increments by $1$. Each outer iteration of the GS algorithm involves the computation of the gradient $\nabla f(\ux_{k-1})$
and a call to the $\ProxS$ procedure to update $x_k$ and $\tilde x_k$. 
Secondly, the $\ProxS$ procedure solves problem~\eqnok{generic_subproblem} iteratively.
Each iteration of this procedure consists
of the computation of subgradient $h'(u_{t-1})$ and the solution of the projection subproblem~\eqnok{def_ut_iter},
which is assumed to be relatively easy to solve (see Section~\ref{sec_prelim}).
For notational convenience, we refer to an iteration of the $\ProxS$ procedure as an inner iteration of the GS algorithm. 
Thirdly, the GS algorithm described above
is conceptual only since we have not yet specified the selection of 
$\{\beta_k\}$, $\{\gamma_k\}$, $\{T_k\}$, $\{p_t\}$ and $\{\theta_t\}$. We
will return to this issue after establishing some convergence properties
of the generic GS algorithm described above.

\vgap

We first present a result which summarizes some important convergence properties
of the $\ProxS$ procedure. 
The following two technical results are needed to establish the convergence
of this procedure.

The first technical result below characterizes
the solution of the projection step \eqnok{def_ut}.
The proof of this result can be found in Lemma 2 of \cite{GhaLan12-2a}.

\begin{lemma} \label{tech1_prox}
Let the convex function $q: X \to \bbr$, the points
$\tx, \ty \in X$ and the scalars $\mu_1, \mu_2 \in \bbr_+$
be given.
Let $\w:\setU \to \bbr$ be a differentiable convex function and
$V(x, z)$ be defined in \eqnok{s450}. If
\[
\hu \in \Argmin \{ q(u) + \mu_1 V(\tx, u) + \mu_2 V(\ty, u) : u \in \setU\},
\]
then for any $ u \in X$, we have
\[
q(\hu) + \mu_1 V(\tx, \hu) + \mu_2 V(\ty, \hu)
 \le q(u) + \mu_1 V(\tx, u) + \mu_2 V(\ty, u) - (\mu_1 + \mu_2) V(\hu, u).
\]
\end{lemma}

The second technical result slightly generalizes Lemma 3 of \cite{Lan13-2}
to provide a convenient way to analyze sequences with sublinear rate of convergence.

\begin{lemma} \label{tech_result_sum} 
Let $w_k\in (0,1]$, $k = 1, 2, \ldots$, and $W_1 > 0$ be given and define
\beq \label{def_Gamma0}
W_k := (1 -\w_k) W_{k-1}, \ \ \ k \ge 2.
\eeq
Suppose that $W_k > 0$ for all $k \ge 2$ and that the sequence $\{\delta_k\}_{k \ge 0}$ satisfies
\beq \label{general_cond}
\delta_k \le (1 -w_k) \delta_{k-1} + B_k, \ \ \ k = 1, 2, \ldots.
\eeq
Then for any $k \ge 1$, we have
\beq  \label{general_tech_result}
\delta_k \le W_k\left[\frac{1-w_1}{W_1} \delta_{0} + \sum_{i=1}^k \frac{B_i}{W_i}\right].
\eeq
\end{lemma} 

\begin{proof}
The result follows from dividing both sides
of \eqnok{general_cond} by $W_k$ and then summing up the resulting inequalities.
\end{proof}

\vgap

We are now ready to establish the convergence of the
$\ProxS$ procedure.
\begin{proposition} \label{tech_inner}
If $\{p_t\}$ and $\{\theta_t\}$ in the $\ProxS$ procedure satisfy
\beq \label{cond_theta}
\theta_t = \frac{P_{t-1} - P_t}{(1-P_t) P_{t-1}}  \ \ \
\mbox{with} \ \ \ 
P_{t} = 
\begin{cases}
1,& t = 0,\\
p_t (1 + p_t)^{-1} P_{t-1}, & t \ge 1,
\end{cases}
\eeq
then, for any $t \ge 1$ and $u \in X$,  
\begin{align} 
\beta (1-P_{t})^{-1} V(u_t, u) &+   [\Phi(\tilde u_t) - \Phi(u) ]
\le \nn\\
&P_{t}(1-P_{t})^{-1} \left[
\beta V(u_{0}, u) +  \frac{M^2}{2\nu \beta} \sum_{i=1}^t( p_i^2 P_{i-1})^{-1}\right],  \label{inner_iter_rec}
\end{align}
where $\Phi$ is defined in \eqnok{generic_subproblem}.
\end{proposition}

\begin{proof}
By \eqnok{nonsmoothh1} and the definition of $l_h$ in \eqnok{def_lh},
we have
$
h(u_t) \le l_h(u_{t-1}, u_t) + M \|u_t - u_{t-1} \| $.
Adding $g(u_t) + \beta V(x, u_t) + \cX(u_t) $ to both sides of this inequality 
and using the definition of $\Phi$ in \eqnok{generic_subproblem}, we obtain
\beq \label{relation_Phi_phi}
\Phi(u_t) \le g(u_t) + l_h(u_{t-1}, u_t) + \beta V(x, u_t) + \cX(u_t) + M \|u_t - u_{t-1} \| .
\eeq
Now applying Lemma~\ref{tech1_prox} to \eqnok{def_ut_iter}, we obtain
\begin{align*}
&g(u_t) + l_h(u_{t-1}, u_t) + \beta V(x, u_t) + \cX(u_t) + \beta p_t V(u_{t-1}, u_t) \\ 
&\le g(u) + l_h(u_{t-1}, u) + \beta V(x, u) + \cX(u) + \beta p_t V(u_{t-1}, u) - \beta (1 + p_t) V(u_t, u)\\
&\le g(u) + h(u) + \beta V(x, u) + \cX(u) + \beta p_t V(u_{t-1}, u) - \beta (1 + p_t) V(u_t, u) \\
&= \Phi(u) + \beta p_t V(u_{t-1}, u) - \beta (1 + p_t) V(u_t, u),
\end{align*}
where the second inequality follows from the convexity of $h$.
Moreover, by the strong convexity of $\w$,
\begin{align*}
- \beta p_t V(u_{t-1}, u_t) + M \| u_t - u_{t-1}\| 
&\le - \frac{\nu \beta p_t }{2} \|u_t - u_{t-1}\|^2 + M \|u_t - u_{t-1}\|  \le \frac{M^2}{2\nu \beta p_t},
\end{align*}
where the last inequality follows from the simple fact that $- a t^2 / 2 + b t \le b^2/(2a)$ for any $a > 0$.
Combining the previous three inequalities, we conclude that
\begin{align*}
\Phi(u_t)  - \Phi(u) & \le  \beta p_t V(u_{t-1}, u) - \beta (1 + p_t) V(u_t, u) + \frac{M^2}{2\nu \beta p_t}.
\end{align*}
Dividing both sides by $1 + p_t$ and rearranging the terms, we obtain
\begin{align*}
\beta V(u_t, u) + \frac{\Phi(u_t)  - \Phi(u)} {1+p_t} &\le \frac{\beta p_t}{1+ p_t} V(u_{t-1}, u) 
+ \frac{M^2}{2\nu \beta (1+p_t) p_t },
\end{align*}
which, in view of the definition of $P_t$ in \eqnok{cond_theta} and 
Lemma~\ref{tech_result_sum} (with $k=t$, $w_k = 1/(1+p_t)$ and $W_k = P_t$), then implies that
\begin{align} 
\frac{\beta}{P_{t}} V(u_t, u) + \sum_{i=1}^t \frac{\Phi(u_i)  - \Phi(u)} {P_{i} (1+p_i)}
&\le \beta V(u_{0}, u) +  \frac{M^2}{2\nu \beta } \sum_{i=1}^t \frac{1}{P_{i} (1+p_i) p_i } \nn\\
&= \beta V(u_{0}, u) +  \frac{M^2}{2\nu \beta } \sum_{i=1}^t (p_i^2 P_{i-1})^{-1}, \label{inner_iter_basic}
\end{align}
where the last identity also follows from the definition of $P_t$ in \eqnok{cond_theta}.
Also note that by the definition of $\tilde u_t$ in the $\ProxS$ procedure and \eqnok{cond_theta}, we have
\[
\tilde u_t 
= \frac{P_{t}}{1-P_{t}} \left(\frac{1- P_{t-1}}{P_{t-1}} \tilde u_{t-1} + \frac{1}{P_{t}(1 + p_t)} u_t \right).
\]
Applying this relation inductively and using the fact that $P_0 = 1$, we can easily see that
\begin{align*}
\tilde u_t &= \frac{P_t}{1- P_t} \left[ \frac{1 - P_{t-2}}{P_{t-2}} \tilde u_{t-2} + \frac{1}{P_{t-1}(1+p_{t-1})} u_{t-1} + \frac{1}{P_t(1+p_t)} u_t\right] \\
& = \ldots 
= \frac{P_{t}}{1-P_{t}} \sum_{i=1}^t  \frac{1}{P_{i}(1 + p_i)} u_i,
\end{align*}
which, in view of the convexity of $\Phi$, then implies that
\beq \label{sum_theta}
\Phi(\tilde u_t) - \Phi(u) \le \frac{P_{t}}{1-P_{t}} \sum_{i=1}^t  \frac{\Phi(u_i) - \Phi(u)}{P_{i}(1 + p_i)} .
\eeq
Combining the above inequality with \eqnok{inner_iter_basic} and rearranging the terms, we obtain \eqnok{inner_iter_rec}.
\end{proof}

\vgap

Setting $u$ to be the optimal solution of \eqnok{generic_subproblem},
we can see that both $x_k$ and $\tilde x_k$ are approximate solutions of
\eqnok{generic_subproblem} if the right hand side (RHS) of \eqnok{inner_iter_rec} is small enough.
With the help of this result, we can establish an important recursion
from which the convergence of the GS algorithm easily follows.

\begin{proposition}
Suppose that $\{p_t\}$ and $\{\theta_t\}$ in the $\ProxS$ procedure
satisfy \eqnok{cond_theta}. Also assume that 
$\{\beta_k\}$ and $\{\gamma_k\}$ in the GS algorithm satisfy 
\beq \label{cond_gamma_beta}
\gamma_1 = 1 \ \ \ \mbox{and} \ \ \ \nu \beta_k - L \gamma_k \ge 0,  \ \ k \ge 1.
\eeq
Then for any $u \in X$ and $k \ge 1$,
\begin{align} 
\Psi(\bx_k) - \Psi(u)  
\le & (1- \gamma_k) [\Psi(\bx_{k-1})  - \Psi(u)] 
 + \gamma_k (1 - P_{T_k})^{-1}  \nn\\
 &\left[
\beta_k V(x_{k-1}, u) - \beta_k  V(x_k, u) + 
 \frac{M^2 P_{T_k}}{2 \nu \beta_k } \sum_{i=1}^{T_k}(p_i^2 P_{i-1})^{-1}\right].\label{GS_recursion}
\end{align}
\end{proposition}

\begin{proof}
First, notice that by the definition of $\bx_k$ and $\ux_k$, we have $\bx_k - \ux_k = \gamma_k (\tilde x_k - x_{k-1})$. 
Using this observation, \eqnok{smoothf1}, the definition of $l_f$ in \eqnok{def_lf},
and the convexity of $f$, we obtain
\begin{align}
f(\bx_k) &\le l_f(\ux_k, \bx_k) + \frac{L}{2} \|\bx_k - \ux_k\|^2 \nn \\
&= (1-\gamma_k) l_f(\ux_k, \bx_{k-1}) + \gamma_k l_f(\ux_k, \tilde x_k) + \frac{L \gamma_k^2}{2} \|\tilde x_k - x_{k-1}\|^2\nn \\
&\le (1-\gamma_k) f(\bx_{k-1}) + \gamma_k \left[ l_f(\ux_k, \tilde x_k) +\beta_k V(x_{k-1}, \tilde x_k) \right] \nn\\
& \qquad- \gamma_k \beta_k V(x_{k-1}, \tilde x_k) + \frac{L \gamma_k^2}{2} \|\tilde x_k - x_{k-1}\|^2\nn \\
&\le (1-\gamma_k) f(\bx_{k-1}) + \gamma_k \left[ l_f(\ux_k, \tilde x_k) +\beta_k V(x_{k-1}, \tilde x_k) \right] \nn\\
& \qquad - \left( \gamma_k \beta_k - \frac{L \gamma_k^2}{\nu} \right) V(x_{k-1}, \tilde x_k) \nn \\
&\le (1-\gamma_k) f(\bx_{k-1}) + \gamma_k \left[ l_f(\ux_k, \tilde x_k) +\beta_k V(x_{k-1}, \tilde x_k) \right], \label{smoothness_f_enhance}
\end{align}
where the third inequality follows from the strong convexity of $\w$ and the last inequality follows from \eqnok{cond_gamma_beta}.
By the convexity of $h$ and $\cX$, we have 
\beq \label{nonsmoothness_h_enhance}
h(\bx_k) + \cX(\bx_k) \le (1-\gamma_k) [h(\bx_{k-1}) + \cX(\bx_{k-1})] + \gamma_k [h(\tilde x_k) + \cX(\tilde x_k)].
\eeq
 Adding up the
previous two inequalities, and using the definitions of $\Psi$ in \eqnok{cp} and $\Phi_k$ in \eqnok{acc2}, we have
\[
\Psi(\bx_k) 
\le (1-\gamma_k) \Psi(\bx_{k-1}) + \gamma_k \Phi_k(\tilde x_k).
\]
Subtracting $\Psi(u)$ from both sides of the above inequality,
we obtain
\begin{align}
\Psi(\bx_k) - \Psi(u) &\le (1- \gamma_k) [\Psi(\bx_{k-1})  - \Psi(u)] + \gamma_k [\Phi_k(\tilde x_k) - \Psi(u)]. \label{bnd_outer}
\end{align}
Also note that by the definition of $\Phi_k$ in \eqnok{acc2} and the convexity of $f$,
\beq \label{lb_big_phi}
\Phi_k(u) \le f(u) + h(u) + \cX(u) + \beta_k V(x_{k-1},u) = \Psi(u) + \beta_k V(x_{k-1}, u), \ \ \forall u \in X.
\eeq
Combining these two inequalities, we obtain
\begin{align}
\Psi(\bx_k) - \Psi(u) &\le (1- \gamma_k) [\Psi(\bx_{k-1})  - \Psi(u)] \nn\\
&\qquad+ \gamma_k [\Phi_k(\tilde x_k) - \Phi_k(u) +  \beta_k V(x_{k-1}, u)]. \label{bnd_outer1}
\end{align}
Now, in view of  \eqnok{inner_iter_rec}, the definition of $\Phi_k$ in \eqnok{acc2},
and the origin of $(x_k, \tilde x_k)$ in \eqnok{def_ut}, we can easily see that,
for any $u \in X$ and $k \ge 1$,
\begin{align*}
\frac{\beta_k}{1-P_{T_k}} V(x_k, u) + & [\Phi_k(\tilde x_k) - \Phi_k(u) ] \le\\
& \qquad \frac{P_{T_k}}{1-P_{T_k}}\left[\beta_k  V(x_{k-1}, u)  
\frac{M^2  }{2\nu \beta_k} \sum_{i=1}^t (p_i^2 P_{i-1})^{-1}\right].
\end{align*}
Plugging the above bound on $\Phi_k(\tilde x_k) - \Phi_k(u) $ into \eqnok{bnd_outer1},
we obtain \eqnok{GS_recursion}.
\end{proof}

\vgap

We are now ready to establish the main convergence properties of the GS algorithm.
Note that the following quantity will be used in our analysis of this algorithm.
\beq \label{def_Gamma}
\Gamma_k = 
\begin{cases}
1, & k =1,\\
(1-\gamma_k) \Gamma_{k-1}, & k \ge 2.
\end{cases}
\eeq

\begin{theorem} \label{main_theorem_d}
Assume that $\{p_t\}$ and $\{\theta_t\}$ in the $\ProxS$ procedure
satisfy \eqnok{cond_theta}, and also that  
$\{\beta_k\}$ and $\{\gamma_k\}$ in the GS algorithm satisfy \eqnok{cond_gamma_beta}. 
\begin{itemize}
\item [a)] If for any $k \ge 2$,
\beq \label{cond_gamma_beta_alpha}
\frac{\gamma_k \beta_k}{\Gamma_k(1 - P_{T_k})} 
\le \frac{\gamma_{k-1} \beta_{k-1}}{\Gamma_{k-1} (1 - P_{T_{k-1}})}, \
\eeq
then we have, for any $N \ge 1$, 
\begin{align} 
\Psi(\bx_N) - \Psi(x^*) \le {\cal B}_d (N) &:= 
 \frac{\Gamma_N \beta_1 }{1 - P_{T_1}}V(x_0, x^*) \nn\\
 & \quad + \frac{M^2 \Gamma_N }{2 \nu} \sum_{k=1}^N 
\sum_{i=1}^{T_k} \frac{\gamma_k P_{T_k}}{\Gamma_k \beta_k (1 - P_{T_k}) p_i^2 P_{i-1}},
\label{main_result_smooth}
\end{align}
where $x^* \in X$ is an arbitrary optimal solution of problem \eqnok{cp},
and $P_t$ and $\Gamma_k$ are defined in \eqnok{def_but} and \eqnok{def_Gamma}, respectively.
\item [b)] If $X$ is compact, and for any $k\ge 2$,
\beq \label{cond_gamma_beta_alpha_bnd}
\frac{\gamma_k \beta_k}{\Gamma_k(1 - P_{T_k})} 
\ge \frac{\gamma_{k-1} \beta_{k-1}}{\Gamma_{k-1} (1 - P_{T_{k-1}})},
\eeq
then \eqnok{main_result_smooth} still holds by simply replacing the first term 
in the definition of ${\cal B}_d(N)$ with 
$
\gamma_{N} \beta_{N} \bar V(x^*)/(1 - P_{T_{N}}),
$
where
$
\bar V(u)  = \max_{x \in X} V(x, u).
$
\end{itemize}
\end{theorem}

\begin{proof}
We conclude from \eqnok{GS_recursion} and  Lemma~\ref{tech_result_sum} that
\begin{align}
\Psi(\bx_N) - \Psi(u)  
&\le \Gamma_N \frac{1-\gamma_1} {\Gamma_1} [\Psi(\bx_0) - \Psi(u)] \nn \\
& \quad + \Gamma_N 
\sum_{k=1}^N \frac{ \beta_k  \gamma_k}{\Gamma_k (1 - P_{T_k})}
\left[
V(x_{k-1}, u) -  V(x_k, u) \right] \nn \\
&\quad + \frac{M^2 \Gamma_N}{2 \nu} \sum_{k=1}^N \sum_{i=1}^{T_k} 
\frac{ \gamma_k P_{T_k}}{\Gamma_k \beta_k (1 - P_{T_k}) p_i^2 P_{i-1}} \nn \\
& = \Gamma_N 
\sum_{k=1}^N \frac{ \beta_k  \gamma_k}{\Gamma_k (1 - P_{T_k})}
\left[
V(x_{k-1}, u) -  V(x_k, u) \right] \nn\\
& \quad+ \frac{M^2 \Gamma_N}{2 \nu} \sum_{k=1}^N \sum_{i=1}^{T_k} 
\frac{ \gamma_k P_{T_k}}{\Gamma_k \beta_k (1 - P_{T_k}) p_i^2 P_{i-1}}, \label{main_iter_rec}
\end{align}
where the last identity follows from the fact that $\gamma_1 = 1$.
Now it follows from \eqnok{cond_gamma_beta_alpha} that
\begin{align}
\sum_{k=1}^N &\frac{\beta_k\gamma_k}{\Gamma_k (1 - P_{T_k})}
\left[
V(x_{k-1}, u) -  V(x_k, u) \right] \nn \\
&\le \frac{\beta_1 \gamma_1  }{\Gamma_1 (1 - P_{T_1})} V(x_0, u) - \frac{\beta_N  \gamma_N }{\Gamma_N (1 - P_{T_N})} V(x_N, u) 
\le \frac{\beta_1}{1 - P_{T_1}} V(x_0, u), \label{bnd_dist1}
 \end{align}
where the last inequality follows from the facts that $\gamma_1 = \Gamma_1 = 1$, $P_{T_N} \le 1$, and $V(x_N, u) \ge0$.
The result in part a) then clearly follows from the previous two inequalities with $u = x^*$.
Moreover, using \eqnok{cond_gamma_beta_alpha_bnd} and the
fact $V(x_k, u) \le \bar V(u) $ , we conclude that
\begin{align}
\sum_{k=2}^N &\frac{\beta_k \gamma_k}{\Gamma_k(1 - P_{T_k})}
\left[
 V(x_{k-1}, u) -  V(x_k, u) \right] \nn \\
&\le \frac{\beta_1}{1-P_{T_1}} \bar V(u) - \sum_{k=2}^N \left[\frac{\beta_{k-1} \gamma_{k-1} }{\Gamma_{k-1} (1 - P_{T_{k-1}})}
- \frac{\beta_{k} \gamma_{k} }{\Gamma_{k} (1 - P_{T_{k}})}
\right] \bar V(u) \nn\\
&= \frac{\gamma_{N} \beta_{N}}{\Gamma_{N} (1 - P_{T_{N}})} \bar V(u). \label{bnd_dist2}
\end{align}
Part b) then follows from the above observation and \eqnok{main_iter_rec} with $u = x^*$.
\end{proof}

\vgap

Clearly, there are various options for specifying the parameters 
$\{p_t\}$, $\{\theta_t\}$, $\{\beta_k\}$, $\{\gamma_k\}$, and $\{T_k\}$
to guarantee the convergence of 
the GS algorithm.
Below we provide a few such selections which lead to
the best possible rate of convergence for solving problem \eqnok{cp}.
In particular, Corollary~\ref{cor_deter_1}.a) provides
a set of such parameters for the case when the feasible region $X$ is 
unbounded and the iteration limit $N$ is given a priori, while the one in Corollary~\ref{cor_deter_1}.b)
works only for the case when $X$ is compact, but does not require
$N$ to be given in advance.

\begin{corollary} \label{cor_deter_1}
Assume that $\{p_t\}$ and $\{\theta_t\}$ in the $\ProxS$ procedure are set to
\beq \label{spec_theta}
p_t = \frac{t}{2} \ \ \ \mbox{and} \ \ \ \theta_t = \frac{2(t+1)}{t(t+3)}, \ \forall \, t \ge 1. 
\eeq
\begin{itemize}
 \item [a)] If $N$ is fixed a priori, and $\{\beta_k\}$, $\{\gamma_k\}$, and 
$\{T_k\}$ are set to
\beq \label{spec_alpha_theta}
\beta_k = \frac{2L}{v k},  \ \
\gamma_k = \frac{2}{k+1}, \ \ \mbox{and} \ \
T_k = \left \lceil \frac{M^2 N  k^2 }{\tilde D L^2} \right \rceil
\eeq
for some $\tilde D > 0$,  
then
\beq \label{main_coro1}
\Psi(\bx_N) - \Psi(x^*) \le \frac{2 L}{N (N+1) } \left[ \frac{3 V(x_0, x^*)}{\nu} + 2 \tilde D \right], \ \ \forall N \ge 1.
\eeq
\item [b)] If $X$ is compact, and
$\{\beta_k\}$, $\{\gamma_k\}$, and $\{T_k\}$ are set to
\beq \label{spec_alpha_theta_bnd}
\beta_k =  \frac{9 L (1- P_{T_k}) }{2 \nu (k+1)}, \ \  
\gamma_k = \frac{3}{k+2}, \ \ \mbox{and} \ \
T_k = \left \lceil \frac{M^2 (k+1)^3 }{\tilde D L^2} \right \rceil, \ \
\eeq
for some $\tilde D > 0$,
then 
\beq \label{main_coro2}
\Psi(\bx_N) - \Psi(x^*) \le   
\frac{L}{(N+1) (N+2)} \left(\frac{27 \bar V(x^*)}{2 \nu} + \frac{8 \tilde D}{3} \right), \ \ \forall N \ge 1.
\eeq
\end{itemize}
\end{corollary}

\begin{proof}
We first show part a).
By the definitions of $P_t$  and $p_t$ in \eqnok{cond_theta} and \eqnok{spec_theta}, we have
\beq \label{compute_Pt}
P_t =  \frac{t P_{t-1}}{t+2} = \ldots = \frac{2}{(t+1) (t+2)}.
\eeq
Using the above identity and \eqnok{spec_theta}, we can easily see that
the condition in \eqnok{cond_theta} holds. It also follows from \eqnok{compute_Pt}
and the definition of $T_k$ in \eqnok{spec_alpha_theta} that
\beq \label{rel_Q}
P_{T_k} \le P_{T_{k-1}} \le \ldots \le P_{T_1} \le \frac{1}{3}.
\eeq 
Now, it can be easily seen from the definition of $\beta_k$ and $\gamma_k$ in \eqnok{spec_alpha_theta}
that \eqnok{cond_gamma_beta} holds.
It also follows from \eqnok{def_Gamma} and \eqnok{spec_alpha_theta}  that
\beq \label{compute_Gamma}
\Gamma_k = \frac{2}{k (k+1)}.
\eeq
By \eqnok{spec_alpha_theta}, \eqnok{rel_Q}, and \eqnok{compute_Gamma}, we have
\[
\frac{\gamma_k \beta_k}{\Gamma_k (1 - P_{T_k})} = \frac{2L}{\nu (1- P_{T_k})}
\le  \frac{2L}{\nu (1- P_{T_{k-1}})} = \frac{\gamma_{k-1} \beta_{k-1}}{\Gamma_{k-1} (1 - P_{T_k-1})},
\]
from which \eqnok{cond_gamma_beta_alpha} follows.
Now, by \eqnok{compute_Pt} and the fact that $p_t = t/2$, we have
\beq \label{rel_Q1}
\sum_{i=1}^{T_k} \frac{1}{p_i^2 P_{i-1}} = 2 \sum_{i=1}^{T_k} \frac{i+1}{i} \le 4 T_k,
\eeq
which, together with \eqnok{spec_alpha_theta} and \eqnok{compute_Gamma}, then imply that
\beq \label{bnd_2nd}
\sum_{i=1}^{T_k} \frac{\gamma_k P_{T_k}}{\Gamma_k \beta_k (1 - P_{T_k}) p_i^2 P_{i-1}}  
\le \frac{4 \gamma_k P_{T_k} T_k}{\Gamma_k \beta_k (1-P_{T_k})}
= \frac{4 \nu k^2}{L(T_k+3)}. 
\eeq
Using this observation, \eqnok{main_result_smooth}, \eqnok{rel_Q}, and \eqnok{compute_Gamma}, we have
\begin{align*}
{\cal B}_d(N) &\le \frac{4 L V(x_0, x^*)}{\nu N (N+1) (1 - P_{T_1})} + \frac{4 M^2}{L N (N+1)} \sum_{k=1}^N \frac{k^2}{T_k+3}\\
&\le \frac{6 L V(x_0, x^*)}{\nu N (N+1) } + \frac{4 M^2}{L N (N+1)} \sum_{k=1}^N \frac{k^2}{T_k+3},
\end{align*}
which, in view of Theorem~\ref{main_theorem_d}.a) and
the definition of $T_k$ in \eqnok{spec_alpha_theta}, then clearly implies
\eqnok{main_coro1}.

Now let us show that part b) holds. It follows from
\eqnok{rel_Q}, and the definition of $\beta_k$ and $\gamma_k$ in \eqnok{spec_alpha_theta_bnd} that
\beq \label{rel_beta}
\beta_k \ge \frac{3 L}{\nu (k+1)} \ge \frac{L \gamma_k}{\nu}
\eeq
and hence that \eqnok{cond_gamma_beta} holds. 
It also follows from \eqnok{def_Gamma} and \eqnok{spec_alpha_theta_bnd} that
\beq \label{compute_Gamma1}
\Gamma_k = \frac{6}{k (k+1) (k+2)}, \ \  k \ge 1,
\eeq
and hence that
\beq \label{rel_beta1}
\frac{\gamma_k \beta_k}{\Gamma_k (1 - P_{T_k})} = \frac{k (k+1)}{2} \frac{9 L}{2\nu (k+1)}
= \frac{9 L k}{ 4 \nu},
\eeq
which implies that \eqnok{cond_gamma_beta_alpha_bnd} holds.
Using \eqnok{spec_alpha_theta_bnd}, \eqnok{rel_Q}, \eqnok{rel_Q1},  and \eqnok{rel_beta}, we have
\begin{align}
\sum_{i=1}^{T_k} \frac{\gamma_k P_{T_k}}{\Gamma_k \beta_k (1 - P_{T_k})p_i^2 P_{i-1}}
& \le  \frac{4 \gamma_k P_{T_k} T_k}{\Gamma_k \beta_k (1-P_{T_k})}
= \frac{4 \nu k (k+1)^2 P_{T_k} T_k}{9 L  (1-P_{T_k})^2} \nn\\
&= \frac{8 \nu k (k+1)^2 (T_k +1) (T_k +2)}{9 L T_k (T_k+3)^2} \le \frac{8\nu k(k+1)^2}{9L T_k}. \label{bnd_2nd_b}
\end{align}
Using this observation, \eqnok{spec_alpha_theta_bnd}, \eqnok{compute_Gamma1}, and Theorem~\ref{main_theorem_d}.b), we conclude that
\begin{align*}
\Psi(\bx_N) - \Psi(x^*) 
&\le \frac{\gamma_N \beta_N \bar V(x^*)}{(1 - P_{T_N})} + \frac{M^2 \Gamma_N} {2\nu} \sum_{k=1}^N \frac{8\nu k(k+1)^2}{9L T_k} \\
&\le \frac{\gamma_N \beta_N \bar V(x^*)}{(1 - P_{T_N})} + \frac{8 L \tilde D}{3  (N+1) (N+2)}\\
&\le  \frac{L}{(N+1) (N+2)} \left(\frac{27 \bar V(x^*)}{2 \nu} + \frac{8 \tilde D}{3} \right).
\end{align*}
\end{proof}

Observe that by \eqnok{def_but} and \eqnok{compute_Pt}, when the selection of $p_t = t/2$, the definition of $\tilde u_t$
in the $\ProxS$ procedure can be simplified as
\[
\tilde u_t = \frac{(t+2)(t-1)}{t(t+3)} \tilde u_{t-1} +  \frac{2(t+1)}{t(t+3)} u_t.
\]

In view of Corollary~\ref{cor_deter_1}, we can establish the complexity of
the GS algorithm for finding an $\epsilon$-solution of problem~\eqnok{cp}.

\begin{corollary} \label{main_bnd_d}
Suppose that $\{p_t\}$ and $\{\theta_t\}$ are set to \eqnok{spec_theta}.
Also assume that there exists an estimate ${\cal D}_X > 0$ s.t. 
\beq \label{def_DX}
V(x, y) \le {\cal D}_X, \ \ \forall x, y \in X.
\eeq
If $\{\beta_k\}$, $\{\gamma_k\}$, and 
$\{T_k\}$ are set to  \eqnok{spec_alpha_theta} with
$
\tilde D = 3 {\cal D}_X/(2 \nu)
$
for some $N > 0$,
then the total number of evaluations for $\nabla f$ and
$h'$ can be bounded by
\beq \label{cor_bnd_grad}
{\cal O} \left( \sqrt{\frac{L {\cal D}_X}{\nu \epsilon}} \right)
\eeq
and
\beq \label{cor_bnd_subgrad}
{\cal O} \left\{\frac{M^2 {\cal D}_X }{\nu \epsilon^2}
+ \sqrt{ \frac{L {\cal D}_X}{\nu \epsilon} } \right\},
\eeq
respectively.
Moreover, the above two complexity bounds also hold if $X$ is bounded, and
$\{\beta_k\}$, $\{\gamma_k\}$, and 
$\{T_k\}$ are set to  \eqnok{spec_alpha_theta} with
$
\tilde D = 81 {\cal D}_X/(16 \nu).
$
\end{corollary}

\begin{proof}In view of Corollary~\ref{cor_deter_1}.a), if $\{\beta_k\}$, $\{\gamma_k\}$, and $\{T_k\}$
are set to \eqnok{spec_alpha_theta}, the total number of outer iterations (or gradient evaluations) performed by
the GS algorithm to find an $\epsilon$-solution of \eqnok{cp} can be bounded by
\beq \label{bnd_outer_N}
N \le \sqrt{\frac{L}{\epsilon} \left[ \frac{3 V(x_0, x^*)}{\nu} + 2 \tilde D \right]} \le \sqrt{\frac{6 L  {\cal D}_X }{\nu \epsilon}}.
\eeq
Moreover, using the definition of $T_k$ in \eqnok{spec_alpha_theta}, we conclude that
the total number of inner iterations (or subgradient evaluations) can be bounded by
\[
\sum_{k=1}^N T_k \le \sum_{k=1}^N \left( \frac{M^2 N  k^2 }{\tilde D L^2} + 1\right) 
\le \frac{M^2 N (N+1)^3 }{3 \tilde D L^2} + N 
=  \frac{2 \nu M^2 N (N+1)^3 }{9 {\cal D}_X L^2} + N, 
\]
which, in view of \eqnok{bnd_outer_N}, then clearly implies the bound in \eqnok{cor_bnd_subgrad}.
Using Corollary~\ref{cor_deter_1}.b) and similar arguments, we can show that 
the complexity bounds \eqnok{cor_bnd_grad}
and \eqnok{cor_bnd_subgrad} also hold when $X$ is bounded, and
$\{\beta_k\}$, $\{\gamma_k\}$, and 
$\{T_k\}$ are set to  \eqnok{spec_alpha_theta} .
\end{proof}

\vgap

In view of Corollary~\ref{main_bnd_d},
the GS algorithm can achieve the optimal complexity bound for solving
problem \eqnok{cp} in terms of the number of evaluations for both $\nabla f$ and
$h'$. To the best of our knowledge, this is the first time that this type of algorithm
has been developed in the literature.

It is also worth noting that we can relax the requirement on ${\cal D}_X$ in \eqnok{def_DX} to
$V(x_0, x^*) \le {\cal D}_X$ 
or $\max_{x \in X} V(x, x^*) \le {\cal D}_X$, respectively,
when the stepsize policies in \eqnok{spec_alpha_theta}
or in \eqnok{spec_alpha_theta_bnd} is used. Accordingly,
we can tighten 
the complexity bounds in \eqnok{cor_bnd_grad}
and \eqnok{cor_bnd_subgrad} by a constant factor.
  
\setcounter{equation}{0}
\section{Stochastic gradient sliding} \label{sec_SGS}
In this section, we consider the situation when the computation of 
stochastic subgradients of $h$ is much easier than that of exact
subgradients. This situation happens, for example, when $h$ is given in the
form of an expectation or as the summation of many nonsmooth components.
By presenting a stochastic gradient sliding (SGS) method,
we show that similar complexity bounds as in Section~\ref{sec_DGS} for solving problem~\eqnok{cp} 
can still be obtained in expectation or with high probability,
but the iteration cost of the SGS method can be substantially 
smaller than that of the GS method.    

More specifically, we assume that the nonsmooth component $h$ is represented by a stochastic oracle ($\SO$)
satisfying  \eqnok{assum_FO_a} and \eqnok{assum_FO_b}. Sometimes, we augment \eqnok{assum_FO_b} by a ``light-tail'' assumption:
\beq \label{assum_FO_c}
\bbe[\exp(\|H(u, \xi) - h'(u) \|_*^2/\sigma^2)] \le \exp(1).
\eeq
It can be easily seen that \eqnok{assum_FO_c} implies \eqnok{assum_FO_b} by Jensen's inequality.

The stochastic gradient sliding (SGS) algorithm is obtained by simply replacing the exact subgradients
in the $\ProxS$ procedure with the stochastic subgradients returned by the $\SO$. 
This algorithm is formally described as follows.

\begin{algorithm} [H]
	\caption{The stochastic gradient sliding (SGS) algorithm}
	\label{algGeneric-s}
	\begin{algorithmic}
\State 
The algorithm is the same as GS except that 
the identity \eqnok{def_ut_iter} in the $\ProxS$ procedure is replaced by
\beq \label{def_ut_iter_s}
u_{t} = \argmin_{u \in X} \left\{ g(u) + \langle H(u_{t-1}, \xi_{t-1}), u \rangle + \beta V(x, u) + \beta p_t V(u_{t-1}, u) + \cX(u) \right\}.
\eeq
The above modified $\ProxS$ procedure is called the $\SProxS$ (stochastic $\ProxS$) procedure.
\end{algorithmic}
\end{algorithm}

We add a few remarks about the above SGS algorithm. Firstly,
in this algorithm, we assume that the exact gradient of $f$ 
will be used throughout the $T_k$ inner iterations. This is different from
the accelerated stochastic approximation in \cite{Lan10-3}, where one 
needs to compute $\nabla f$ at each subgradient projection step. Secondly,
let us denote 
\beq \label{def_th}
\tilde l_h(u_{t-1}, u) := h(u_{t-1}) + \langle H(u_{t-1}, \xi_{t-1}), u - u_{t-1}\rangle. 
\eeq
It can be easily seen that \eqnok{def_ut_iter_s}
is equivalent to 
\beq \label{def_ut_iter_s1}
u_{t} = \argmin_{u \in X} \left\{ g(u) + \tilde l_h(u_{t-1}, u) + \beta V(x, u) + \beta p_t V(u_{t-1}, u) + \cX(u) \right\}.
\eeq
This problem reduces to \eqnok{def_ut_iter} if there is no stochastic noise associated with
the $\SO$, i.e., $\sigma = 0$ in \eqnok{assum_FO_b}.
Thirdly, note that we have not provided the specification of $\{\beta_k\}$,
$\{\gamma_k\}$, $\{T_k\}$, $\{p_t\}$ and $\{\theta_t\}$ in the SGS algorithm.
Similarly to Section~\ref{sec_DGS}, we will return to this issue after
establishing some convergence properties about the generic $\SProxS$ procedure
and SGS algorithm.
 



\vgap

The following result describes some important convergence properties
of the $\SProxS$ procedure.

\begin{proposition} \label{tech_inner_s}
Assume that $\{p_t\}$ and $\{\theta_t\}$ in the $\SProxS$ procedure
satisfy \eqnok{cond_theta}.
Then for any $t \ge 1$ and $u \in X$,
\begin{align}
\beta (1 - &P_{t})^{-1} V(u_t, u) + [\Phi(\tilde u_t)  - \Phi(u)]
\le \beta P_{t} (1 - P_{t})^{-1} V(u_{t-1}, u)  \, \, +  \nn \\
& P_{t} (1 - P_{t})^{-1} \sum_{i=1}^{t} (p_i P_{i-1})^{-1}
\left[ \frac{\left(M + \|\delta_{i}\|_*\right)^2}{2\nu \beta p_i }
+ \langle \delta_i, u - u_{i-1} \rangle \right], \label{inner_iter_rec_s}
\end{align}
where $\Phi$ is defined in \eqnok{generic_subproblem},
\beq \label{def_delta_t}
\delta_{t} := H(u_{t-1}, \xi_{t-1}) - h'(u_{t-1}), \ \ \mbox{and}  \ \ h'(u_{t-1}) = \bbe[H(u_{t-1}, \xi_{t-1})].
\eeq
\end{proposition}

\begin{proof}
Let 
$
\tilde l_h(u_{t-1}, u)$ be defined in \eqnok{def_th}.
Clearly, we have $\tilde l_h(u_{t-1}, u) - l_h(u_{t-1}, u) = \langle \delta_t, u - u_{t-1} \rangle$.
Using this observation and \eqnok{relation_Phi_phi}, we obtain
\begin{align*}
\Phi(u_t) &\le g(u) + l_h(u_{t-1}, u_t) + \beta V(x, u_t) + \cX(u_t) + M \|u_t - u_{t-1} \|\\
&= g(u) + \tilde l_h(u_{t-1}, u_t) - \langle \delta_t, u_t- u_{t-1} \rangle 
+ \beta V(x, u_t) + \cX(u_t) + M \|u_t - u_{t-1}\|\\
&\le g(u) + \tilde l_h(u_{t-1}, u_t) + \beta V(x, u_t) + \cX(u_t) + (M + \|\delta_t\|_*) \|u_t - u_{t-1}\|,
\end{align*}
where the last inequality follows from the Cauchy-Schwarz inequality.
Now applying Lemma~\ref{tech1_prox} to \eqnok{def_ut_iter_s}, we obtain
\begin{align*}
&g(u_t) + \tilde l_h(u_{t-1}, u_t) + \beta V(x, u_t)+ \beta p_t V(u_{t-1}, u_t) + \cX(u_t)  \\
&\le g(u) + \tilde l_h(u_{t-1}, u) + \beta V(x, u) + \beta p_t V(u_{t-1}, u)  + \cX(u) - \beta(1 + p_t) V(u_t, u)\\
&= g(u) + l_h(u_{t-1}, u) + \langle \delta_t, u - u_{t-1} \rangle \\
&\qquad+ \beta V(x, u)+ \beta p_t V(u_{t-1}, u)  + \cX(u) - \beta(1 + p_t) V(u_t, u)\\
&\le \Phi(u) + \beta p_t V(u_{t-1}, u) - \beta (1 + p_t) V(u_t, u) + \langle \delta_t, u - u_{t-1} \rangle,
\end{align*}
where the last inequality follows from the convexity of $h$ and \eqnok{generic_subproblem}.
Moreover, by the strong convexity of $\w$,
\begin{align*}
- \beta p_t V(u_{t-1}, u_t) &+ (M+\|\delta_t\|_*) \| u_t - u_{t-1}\| \\
&\le - \frac{\nu \beta p_t }{2} \|u_t - u_{t-1}\|^2 + (M + \|\delta_t\|_*) \|u_t - u_{t-1}\|  
\le \frac{\left(M+\|\delta_t\|_*\right)^2}{2\nu \beta p_t},
\end{align*}
where the last inequality follows from the simple fact that $- a t^2 / 2 + b t \le b^2/(2a)$ for any $a > 0$.
Combining the previous three inequalities, we conclude that
\begin{align*}
\Phi(u_t)  - \Phi(u) & \le  \beta p_t V(u_{t-1}, u) - \beta (1 + p_t) V(u_t, u) 
+ \frac{\left(M + \|\delta_t\|_*\right)^2}{2\nu \beta p_t}
+ \langle \delta_t, u - u_{t-1} \rangle.
\end{align*}
Now dividing both sides of the above inequality by $1 + p_t$ and re-arranging the terms, we obtain
\begin{align*}
\beta V(u_t, u) + \frac{\Phi(u_t)  - \Phi(u)} {1+p_t} 
&\le \frac{\beta p_t}{1+ p_t} V(u_{t-1}, u) + \frac{\left(M + \|\delta_t\|_*\right)^2}{2\nu \beta (1+p_t) p_t }
+ \frac{\langle \delta_t, u - u_{t-1} \rangle}{1 + p_t},
\end{align*}
which, in view of Lemma~\ref{tech_result_sum}, then implies that
\begin{align} 
\frac{\beta}{P_t} V(u_t, u) + \sum_{i=1}^t \frac{\Phi(u_i)  - \Phi(u)} {P_i (1+p_i)}
&\le \beta V(u_{0}, u) \nn\\
&+  \sum_{i=1}^t \left[ \frac{\left(M + \|\delta_i\|_*\right)^2}{2\nu \beta P_i (1+p_i) p_i } 
+ \frac{\langle \delta_i, u - u_{i-1} \rangle}{P_i (1 + p_i)} \right]. \label{inner_iter_basic_s}
\end{align}
The result then immediately follows from the above inequality and \eqnok{sum_theta}.
\end{proof}

\vgap

It should be noted that the 
search points $\{u_t\}$ generated by different calls to the $\SProxS$ procedure
in different outer iterations of the SGS algorithm
are distinct from each other. To avoid ambiguity, we use  
$u_{k,t}$, $k \ge 1$, $t\ge 0$, to denote the search points generated by the $\SProxS$ procedure
in the $k$-th outer iteration.
Accordingly, we use 
\beq \label{def_delta_kt}
\delta_{k,t-1} := H(u_{k, t-1}, \xi_{t-1}) - h'(u_{k, t-1}), \ \ k \ge 1, t \ge 1,
\eeq
to denote the stochastic noises associated with the $\SO$.
Then, by \eqnok{inner_iter_rec_s}, the definition of $\Phi_k$ in \eqnok{acc2}, and
the origin of $(x_k, \tilde x_k)$ in the SGS algorithm, we have
\begin{align}
\beta_k(1& - P_{T_k})^{-1} V(x_k, u) + [\Phi_k(\tilde x_k)  - \Phi_k(u)]
\le \beta_k P_{T_k} (1 - P_{T_k})^{-1} V(x_{k-1}, u)  \, \, +  \nn \\
& P_{T_k} (1 - P_{T_k})^{-1} \sum_{i=1}^{T_k} \frac{1}{p_i P_{i-1}}
\left[ \frac{\left(M + \|\delta_{k,i-1}\|_*\right)^2}{2\nu \beta_k p_i }
+ \langle \delta_{k,i-1}, u - u_{k,i-1} \rangle \right] \label{inner_iter_rec_s_v}
\end{align}
for any $u \in X$ and $k \ge 1$.

\vgap

With the help of \eqnok{inner_iter_rec_s_v},
we are now ready to establish the main convergence properties of the
SGS algorithm. 

\begin{theorem} \label{main_theorem_s}
Suppose that $\{p_t\}$, $\{\theta_t\}$, $\{\beta_k\}$, and $\{\gamma_k\}$ in the SGS algorithm
satisfy \eqnok{cond_theta} and \eqnok{cond_gamma_beta}. 
\begin{itemize}
\item [a)] If relation \eqnok{cond_gamma_beta_alpha} holds,
then under Assumptions \eqnok{assum_FO_a} and \eqnok{assum_FO_b},
we have, for any $N \ge 1$, 
\begin{align} 
\bbe\left[\Psi(\bx_N) - \Psi(x^*)\right] 
\le \tilde {\cal B}_d(N) &:= \frac{\Gamma_N \beta_1}{1 - P_{T_1}} V(x_0, u) \nn\\
 &\quad + \frac{\Gamma_N}{\nu} \sum_{k=1}^N  
\sum_{i=1}^{T_k} \frac{(M^2 + \sigma^2)\gamma_k P_{T_k}}{\beta_k \Gamma_k (1 - P_{T_k}) p_i^2 P_{i-1}},
\label{main_result_smooth_exp}
\end{align}
where
$x^*$ is an arbitrary optimal solution of \eqnok{cp}, and
$P_t$ and $\Gamma_k$ are defined in \eqnok{def_but} and \eqnok{def_Gamma},
respectively.
\item [b)] If in addition, $X$ is compact and Assumption~\eqnok{assum_FO_c} holds, then
\beq \label{main_result_smooth_prob}
\prob \left\{\Psi(\bx_N) - \Psi(x^*) \ge \tilde {\cal B}_d(N) + \lambda  {\cal B}_p(N) \right\} \le \exp\left\{-2\lambda^2/3\right\}
+ \exp\left\{-\lambda\right\},
\eeq

for any $\lambda > 0$ and $N \ge 1$, where
\begin{align} 
\tilde {\cal B}_p(N) &:= \sigma \Gamma_N \left\{\frac {2 \bar V(x^*)}{\nu} \sum_{k=1}^N \sum_{i=1}^{T_k}
\left[\frac{\gamma_k P_{T_k}}{\Gamma_k (1 - P_{T_k}) p_i P_{i-1}}\right]^2 \right\}^\frac{1}{2} \nn \\
& \quad+ \frac{\Gamma_N}{\nu} \sum_{k=1}^N \sum_{i=1}^{T_k}  \frac{\sigma^2 \gamma_k P_{T_k} }{\beta_k \Gamma_k (1 - P_{T_k}) p_i^2 P_{i-1}}.
\label{def_BP}
\end{align}
\item [c)] If $X$ is compact and relation \eqnok{cond_gamma_beta_alpha_bnd} (instead of \eqnok{cond_gamma_beta_alpha})
holds, then both part a) and part b) still hold by replacing the first term
in the definition of   $\tilde {\cal B}_d(N)$ with $
\gamma_{N} \beta_{N} \bar V(x^*)/(1 - P_{T_{N}}).
$
\end{itemize}
\end{theorem}

\begin{proof}
Using \eqnok{bnd_outer1} and \eqnok{inner_iter_rec_s_v}, we have
\begin{align*}
\Psi(\bx_k) - \Psi(u)  
&\le  (1- \gamma_k) [\Psi(\bx_{k-1})  - \Psi(u)]
+ \gamma_k \left\{
\frac{\beta_k}{1 - P_{T_k}} [V(x_{k-1}, u) - V(x_k, u)] + \right. \\
 & \left. \frac{P_{T_k}}{1 - P_{T_k} } 
\sum_{i=1}^{T_k} \frac{1}{p_i P_{i-1}} \left[\frac{\left(M + \|\delta_{k,i-1}\|_*\right)^2}{2\nu \beta_k p_i } 
+ \langle \delta_{k,i-1}, u - u_{k,i-1} \rangle \right]\right\}.
\end{align*}
Using the above inequality and Lemma~\ref{tech_result_sum}, we conclude that
\begin{align}
\Psi&(\bx_N) - \Psi(u)   \le \Gamma_N (1- \gamma_1) [\Psi(\bx_{0})  - \Psi(u)] \nn\\
&+ \Gamma_N 
\sum_{k=1}^N \frac{\beta_k \gamma_k}{\Gamma_k (1 - P_{T_k})}
\left[
V(x_{k-1}, u) - V(x_k, u) \right] + \Gamma_N \sum_{k=1}^N 
 \frac{\gamma_k P_{T_k}}{\Gamma_k (1 - P_{T_k}) } 
\nn \\
& \quad \sum_{i=1}^{T_k} \frac{1}{p_i P_{i-1}} \left[\frac{\left(M + \|\delta_{k,i-1}\|_*\right)^2}{2\nu \beta_k p_i } 
+ \langle \delta_{k,i-1}, u - u_{k,i-1} \rangle \right]. \nn
\end{align}
The above relation, in view of \eqnok{bnd_dist1} and the fact that $\gamma_1 = 1$, then implies that
\begin{align} 
\Psi(\bx_N) - \Psi(u) &\le \frac{\beta_k}{1 - P_{T_1}} V(x_0, u)  + \Gamma_N \sum_{k=1}^N  
 \frac{\gamma_k P_{T_k}}{\Gamma_k (1 - P_{T_k}) } \nn \\
& \quad
\sum_{i=1}^{T_k} \frac{1}{p_i P_{i-1}} \left[\frac{M^2 + \|\delta_{k,i-1}\|_*^2}{\nu \beta_k p_i } 
+ \langle \delta_{k,i-1}, u - u_{k,i-1} \rangle \right]. \label{stoch_main_rec_ab}
\end{align}
We now provide bounds on the RHS of \eqnok{stoch_main_rec_ab} in expectation or with high probability.

We first show part a). Note that by our assumptions on the $\SO$, the random variable $\delta_{k,i-1}$ is independent of
the search point $u_{k, i-1}$ and hence $\bbe[\langle \Delta_{k, i-1},  x^* - u_{k, i}\rangle]=0$.
In addition, Assumption \ref{assum_FO_b} implies that $\bbe[\|\delta_{k,i-1}\|_*^2]\leq \sigma^2$. Using
the previous two observations and taking expectation on both sides of \eqnok{stoch_main_rec_ab} (with $u = x^*$), 
we obtain \eqnok{main_result_smooth_exp}.

We now show that part b) holds. Note that by our assumptions on the $\SO$
and the definition of $u_{k,i}$, the sequence $\{\langle \delta_{k,i-1}, x^* - u_{k, i-1}\rangle\}_{k \ge 1, 1 \le i \le T_k}$
is a martingale-difference sequence. 
Denoting
\[
\alpha_{k,i} := \frac{\gamma_k P_{T_k}}{\Gamma_k (1 - P_{T_k}) p_i P_{i-1}},
\]
and
using the large-deviation
theorem for martingale-difference sequence (e.g., Lemma 2 of \cite{lns11})
and the fact that
	\begin{align*}
				& \bbe\left[\exp\left\{\nu \alpha_{k,i}^2 \langle \delta_{k,i-1},  x^* - u_{k,i}\rangle^2 / \left(2\alpha_{k,i}^2 \bar V(x^*)\sigma^2 \right) \right\} \right]\\
				&\leq \bbe\left[\exp\left\{\nu \alpha_{k,i}^2\|\delta_{k,i-1}\|_*^2\| x^* - u_{k,i}\|^2 / \left(2\bar V(x^*)\sigma^2 \right) \right\} \right]\\
				&\leq \bbe\left[\exp\left\{\|\delta_{k,i-1}\|_*^2 V( u_{k,i},  x^*) / \left(\bar V(x^*)\sigma^2 \right) \right\} \right]\\
				&\leq  \bbe\left[\exp\left\{\|\delta_{k,i-1}\|_*^2/ \sigma^2 \right\} \right] \leq \exp\{1\},
			\end{align*}
we conclude that
\beq \label{bnd_prob1}
\begin{array}{l}
\prob\left\{\sum_{k=1}^N \sum_{i=1}^{T_k}\alpha_{k,i}\langle\delta_{k,i-1}, x^* - u_{k, i-1}\rangle
>\lambda \sigma \sqrt{\frac {2 \bar V(x^*)}{\nu} \sum_{k=1}^N \sum_{i=1}^{T_k}\alpha_{k,i}^2 }  \right\}\\
\leq \exp\{-\lambda^2/3 \}, \forall \lambda >0.
\end{array}
\eeq
Now let 
\[
S_{k,i} := \frac{\gamma_k P_{T_k}}{\beta_k \Gamma_k (1 - P_{T_k}) p_i^2 P_{i-1}}
\]
and $S:=\sum_{k=1}^N \sum_{i=1}^{T_k}S_{k,i}$. By the convexity of exponential function, we have
	\[
			\begin{array}{l}
				\bbe\left[\exp\left\{\frac 1S \sum_{k=1}^N \sum_{i=1}^{T_k}
				S_{k,i}{\|\delta_{k,i}\|_*^2}/{\sigma^2} \right\}\right] \\
				\leq \bbe\left[\frac 1S \sum_{k=1}^N \sum_{i=1}^{T_k}S_i
				\exp\left\{\|\delta_{k,i}\|_*^2/\sigma^2 \right\}\right]\leq\exp\{1\}.
			\end{array}
\]
where the last inequality follows from Assumption \ref{assum_FO_c}. Therefore, by Markov's inequality, for all $\lambda>0$,
\beq \label{bnd_prob2}
			\begin{array}{l}
				\prob\left\{ \sum_{k=1}^N \sum_{i=1}^{T_k} S_{k,i} \|\delta_{k,i-1}\|_*^2 > 
				(1+\lambda)\sigma^2 \sum_{k=1}^N \sum_{i=1}^{T_k} S_{k,i}\right\} \\
				=  \prob\left\{\exp\left\{\frac 1S \sum_{k=1}^N \sum_{i=1}^{T_k}S_{k,i}
				{\|\delta_{k,i-1}\|_*^2}/{\sigma^2} \right\}\geq \exp\{1+\lambda\} \right\}
				\leq 	\exp\{-\lambda\}.
			\end{array}
\eeq
Our result now directly follows from \eqnok{stoch_main_rec_ab}, \eqnok{bnd_prob1} and \eqnok{bnd_prob2}.
The proof of part c) is very similar to part a) and b) in view of the bound in \eqnok{bnd_dist2},
and hence the details are skipped.
\end{proof}

\vgap

We now provide some specific choices for the parameters $\{\beta_k\}$, $\{\gamma_k\}$,
$\{T_k\}$, $\{p_t\}$, and $\{\theta_t\}$ used in the SGS algorithm. In particular, while the stepsize policy
in Corollary~\ref{stoch_corol1}.a) requires the number of iterations $N$
given a priori, such an assumption is not needed in Corollary~\ref{stoch_corol1}.b) given that $X$
is bounded. However, in order to provide some large-deviation results associated with
the rate of convergence for the SGS algorithm (see \eqnok{stoch_corol1_b} and
\eqnok{stoch_corol2_b} below), we need to assume the boundness of $X$ 
in both Corollary~\ref{stoch_corol1}.a) and Corollary~\ref{stoch_corol1}.b).

\begin{corollary} \label{stoch_corol1}
Assume that $\{p_t\}$ and $\{\theta_t\}$ in the $\SProxS$ procedure are set to \eqnok{spec_theta}.
\begin{itemize}
\item [a)] If $N$ is given a priori, $\{\beta_k\}$ and $\{\gamma_k\}$ are set to \eqnok{spec_alpha_theta}, and
$\{T_k\}$ is given by
\beq \label{def_TK_S}
T_k = \left \lceil \frac{N (M^2 + \sigma^2) k^2}{\tilde D L^2} \right \rceil
\eeq
for some $\tilde D > 0$. Then under 
Assumptions \eqnok{assum_FO_a} and \eqnok{assum_FO_b}, we have
\beq \label{stoch_corol1_a}
\bbe\left[\Psi(\bx_N) - \Psi(x^*)\right] \le \frac{2 L}{N (N+1) } \left[ \frac{3 V(x_0, x^*)}{\nu} + 4 \tilde D \right], \ \ \forall N \ge 1.
\eeq
If in addition, $X$ is compact and Assumption \eqnok{assum_FO_c} holds, then
\begin{align}
\prob &\left\{\Psi(\bx_N) - \Psi(x^*) \ge
\frac{2 L}{N (N+1) } \left[ \frac{3 V(x_0, x^*)}{\nu} + 4 (1 + \lambda) \tilde D
+ \frac{4 \lambda \sqrt{\tilde D \bar V(x^*)}}{\sqrt{3 \nu}}\right]\right\} \nn \\ 
&\le \exp\left\{-2\lambda^2/3\right\}
+ \exp\left\{-\lambda\right\}, \ \forall \lambda > 0, \, \forall N \ge 1.  \label{stoch_corol1_b}
\end{align}
\item [b)] If $X$ is compact, $\{\beta_k\}$ and $\{\gamma_k\}$ are set to
\eqnok{spec_alpha_theta_bnd}, and $\{T_k\}$ is given by
\beq \label{def_TK_S_bnd}
T_k = \left \lceil \frac{(M^2 + \sigma^2) (k+1)^3}{\tilde D L^2} \right \rceil
\eeq
for some $\tilde D > 0$. Then under 
Assumptions \eqnok{assum_FO_a} and \eqnok{assum_FO_b}, we have
\beq \label{stoch_corol2_a}
\bbe\left[\Psi(\bx_N) - \Psi(x^*)\right] \le \frac{L}{(N+1)(N+2) } \left[ \frac{27  \bar V(x^*)}{2 \nu} + \frac{16 \tilde D}{3} \right], \ \ \forall N \ge 1.
\eeq
If in addition, Assumption \eqnok{assum_FO_c} holds, then
\begin{align}
\prob &\left\{\Psi(\bx_N) - \Psi(x^*) \ge
\frac{L}{N (N+2) } \left[ \frac{27 \bar V(x^*)}{2\nu} + \frac{8}{3} (2 + \lambda) \tilde D
+ \frac{12 \lambda \sqrt{2 \tilde D \bar V(x^*)}}{\sqrt{3 \nu}}\right]\right\} \nn \\ 
&\le \exp\left\{-2\lambda^2/3\right\}
+ \exp\left\{-\lambda\right\}, \ \forall \lambda > 0, \, \forall N \ge 1. \label{stoch_corol2_b}
\end{align}
\end{itemize}

\end{corollary}

\begin{proof}
We first show part a). It can be easily seen from \eqnok{compute_Gamma} that \eqnok{cond_gamma_beta} holds. Moreover, 
Using \eqnok{spec_alpha_theta}, \eqnok{rel_Q}, and \eqnok{compute_Gamma}, we can easily see that \eqnok{cond_gamma_beta_alpha} holds.
By  \eqnok{rel_Q}, \eqnok{compute_Gamma}, \eqnok{bnd_2nd},  \eqnok{main_result_smooth_exp}, and \eqnok{def_TK_S}, we have
\begin{align}
\tilde {\cal B}_d(N) &\le \frac{4 L V(x_0, x^*)}{\nu N (N+1) (1 - P_{T_1})} + \frac{8\left(M^2+ \sigma^2\right)}{L N (N+1)} \sum_{k=1}^N \frac{k^2}{T_k+3}\nn \\
&\le \frac{6 L}{\nu N (N+1) } + \frac{8\left(M^2+\sigma^2\right)}{L N (N+1)} \sum_{k=1}^N \frac{k^2}{T_k+3} \nn \\
&\le \frac{2L}{N (N+1) } \left[ \frac{3 V(x_0, x^*)}{\nu} + 4 \tilde D \right], \label{bnd_Bd_s}
\end{align}
which, in view of Theorem~\ref{main_theorem_s}.a), then clearly implies
\eqnok{stoch_corol1_a}.
Now observe that by the definition of $\gamma_k$ in \eqnok{spec_alpha_theta} and relation \eqnok{compute_Gamma},
\begin{align*}
\sum_{i=1}^{T_k} &
\left[\frac{\gamma_k P_{T_k}}{\Gamma_k (1 - P_{T_k}) p_i P_{i-1}}\right]^2 
= \left( \frac{2k}{T_k(T_k+3)}\right)^2  \sum_{i=1}^{T_k} (i+1)^2 \nn \\
&= \left( \frac{2k}{T_k(T_k+3)}\right)^2 \frac{(T_k +1) (T_k +2) (2 T_k + 3)}{6} 
\le \frac{8 k^2}{3 T_k},
\end{align*}
which together with \eqnok{compute_Gamma}, \eqnok{bnd_2nd}, and \eqnok{def_BP} then imply that
\begin{align*}
\tilde {\cal B}_p(N) &\le \frac{2 \sigma}{N (N+1)} \left[\frac{2 \bar V(x^*)}{\nu} \sum_{k=1}^N \frac{8 k^2}{3 T_k} \right]^\frac{1}{2}
+ \frac{8 \sigma^2}{L N (N+1)} \sum_{k=1}^N \frac{k^2}{T_k+3}\\
&\le \frac{2 \sigma}{N(N+1)}\left[\frac{16 \tilde D L^2 \bar V(x^*)}{3 \nu (M^2 + \sigma^2)}\right]^\frac{1}{2}
+ \frac{8 \tilde D L \sigma^2}{N(N+1) (M^2 + \sigma^2)}\\
&\le \frac{8 L}{N (N+1)} \left(\frac{\sqrt{\tilde D \bar V(x^*)}}{\sqrt{3 \nu}} + \tilde D \right).
\end{align*}
Using the above inequality, \eqnok{bnd_Bd_s}, Theorem~\ref{main_theorem_s}.b), we obtain
\eqnok{stoch_corol1_b}.

We now show that part b) holds.
Note that $P_t$ and $\Gamma_k$ are given by \eqnok{compute_Pt} and \eqnok{compute_Gamma1}, respectively.
It then follows from \eqnok{rel_beta} and \eqnok{rel_beta1} that both \eqnok{cond_gamma_beta} and \eqnok{cond_gamma_beta_alpha_bnd} hold.
Using \eqnok{bnd_2nd_b}, the definitions of $\gamma_k$ and $\beta_k$ in \eqnok{spec_alpha_theta_bnd}, 
\eqnok{def_TK_S_bnd}, and Theorem~\ref{main_theorem_s}.c), we conclude that
\begin{align}
\bbe\left[\Psi(\bx_N) - \Psi(x^*)\right] 
&\le \frac{\gamma_N \beta_N \bar V(x^*)}{(1 - P_{T_N})} + \frac{\Gamma_N(M^2 + \sigma^2)}{\nu} \sum_{k=1}^N  
\sum_{i=1}^{T_k} \frac{\gamma_k P_{T_k}}{\beta_k \Gamma_k (1 - P_{T_k}) p_i^2 P_{i-1}} \nn \\
&\le \frac{\gamma_N \beta_N \bar V(x^*)}{(1 - P_{T_N})} + \frac{16 L \tilde D}{3 \nu  (N+1) (N+2)} \nn \\
&\le  \frac{L}{(N+1) (N+2)} \left(\frac{27 \bar V(x^*)}{2 \nu} + \frac{16 \tilde D}{3} \right). \label{bnd_exp_cor2}
\end{align}
Now observe that by the definition of $\gamma_k$ in \eqnok{spec_alpha_theta_bnd}, 
the fact that $p_t = t/2$,  \eqnok{compute_Pt},
 and \eqnok{compute_Gamma1}, we have
\begin{align*}
\sum_{i=1}^{T_k} &
\left[\frac{\gamma_k P_{T_k}}{\Gamma_k (1 - P_{T_k}) p_i P_{i-1}}\right]^2 
= \left( \frac{k (k+1)}{T_k(T_k+3)}\right)^2  \sum_{i=1}^{T_k} (i+1)^2 \nn \\
&= \left( \frac{k (k+1)}{T_k(T_k+3)}\right)^2 \frac{(T_k +1) (T_k +2) (2 T_k + 3)}{6} 
\le \frac{8 k^4}{3 T_k},
\end{align*}
which together with \eqnok{compute_Gamma1}, \eqnok{bnd_2nd_b}, and \eqnok{def_BP} then imply that
\begin{align*}
\tilde {\cal B}_p(N) &\le 
\frac{6}{N(N+1)(N+2)} \left[ 
\sigma \left( \frac{2 \bar V(x^*)}{\nu} \sum_{k=1}^N \frac{8 k^4}{3 T_k} \right)^\frac{1}{2} +
\frac{4 \sigma^2}{9 L} \sum_{k=1}^N \frac{k (k+1)^2}{T_k} \right]\\
&= \frac{6}{N(N+1)(N+2)} \left[ 
\sigma \left( \frac{8 \bar V(x^*) \tilde D L^2 N(N+1)}{3 \nu (M^2 +\sigma^2)} \right)^\frac{1}{2} +
\frac{4 \sigma^2 L \tilde D N}{9  (M^2 + \sigma^2)}\right]\\
&\le \frac{6 L}{N(N+2)} \left( \frac{2\sqrt{2 \bar V(x^*) \tilde D}}{\sqrt{3 \nu}} + \frac{4 \tilde D}{9} \right) .
\end{align*}
The relation in \eqnok{stoch_corol2_b} then immediately follows from the above inequality, \eqnok{bnd_exp_cor2}, and Theorem~\ref{main_theorem_s}.c).
\end{proof}

\vgap

Corollary~\ref{cor_bnd_s} below states the complexity of 
the SGS algorithm for finding a stochastic $\epsilon$-solution of \eqnok{cp}, i.e., a point $\bar x \in X$ s.t.
$\bbe[\Psi(\bar x) - \Psi^*] \le \epsilon$ for some $\epsilon > 0$,
as well as a stochastic $(\epsilon, \Lambda)$-solution of \eqnok{cp}, i.e.,
a point $\bar x \in X$ s.t. $\prob\left\{ \Psi(\bar x) - \Psi^* \le \epsilon \right\}  > 1 - \Lambda$
for some $\epsilon > 0$ and $\Lambda \in (0,1)$.
Since this result follows as an immediate consequence of Corollary~\ref{stoch_corol1}, we skipped
the details of its proof.

\begin{corollary} \label{cor_bnd_s}
Suppose that $\{p_t\}$ and $\{\theta_t\}$ are set to \eqnok{spec_theta}.
Also assume that there exists an estimate ${\cal D}_X > 0$ s.t. \eqnok{def_DX} holds.

\begin{itemize}
\item [a)] If 
$\{\beta_k\}$ and $\{\gamma_k\}$ are set to \eqnok{spec_alpha_theta}, and
$\{T_k\}$ is given by \eqnok{def_TK_S} with
$
\tilde D = 3 {\cal D}_X/(4 \nu)
$
for some $N > 0$,
then the number of evaluations for $\nabla f$ and $h'$, respectively, required by the SGS algorithm
to find a stochastic $\epsilon$-solution of \eqnok{cp} can be bounded by
\beq \label{cor_bnd_grad_s}
{\cal O} \left( \sqrt{\frac{L {\cal D}_X}{\nu \epsilon}} \right)
\eeq
and
\beq \label{cor_bnd_subgrad_s}
{\cal O} \left\{\frac{(M^2+\sigma^2) {\cal D}_X }{\nu \epsilon^2}
+ \sqrt{ \frac{L {\cal D}_X}{\nu \epsilon} } \right\}.
\eeq

\item [b)] If in addition, Assumption \eqnok{assum_FO_c} holds,
then 
the number of evaluations for $\nabla f$ and
$h'$, respectively, required by the SGS algorithm
to find a stochastic $(\epsilon,\Lambda)$-solution of \eqnok{cp} can be bounded by
\beq \label{cor_bnd_grad_sp}
{\cal O} \left\{ \sqrt{\frac{L {\cal D}_X}{\nu \epsilon}\max \left(1, \log \frac{1}{\Lambda} \right) }\right\}
\eeq
and
\beq \label{cor_bnd_subgrad_sp}
{\cal O} \left\{\frac{M^2 {\cal D}_X }{\nu \epsilon^2} \max \left(1, \log^2 \frac{1}{\Lambda} \right)
+ \sqrt{ \frac{L {\cal D}_X}{\nu \epsilon} \max \left(1, \log \frac{1}{\Lambda} \right)} \right\}.
\eeq

\item [c)] The above bounds in part a) and b) still hold if $X$ is bounded, $\{\beta_k\}$ and $\{\gamma_k\}$ are set to
\eqnok{spec_alpha_theta_bnd}, and $\{T_k\}$ is given by \eqnok{def_TK_S_bnd}
with
$
\tilde D = 81 {\cal D}_X/(32 \nu).
$
\end{itemize}
\end{corollary}

Observe that both bounds in \eqnok{cor_bnd_grad_s}
and \eqnok{cor_bnd_subgrad_s} on the number of evaluations for $\nabla f$ and $h'$ are essentially not improvable. In fact,
to the best of our knowledge, this is the first time that the ${\cal O}(1/\sqrt{\epsilon})$
complexity bound on gradient evaluations has been established in the literature
for stochastic approximation type algorithms applied to solve the composite problem in \eqnok{cp}.


\setcounter{equation}{0}
\section{Generalization to strongly convex and structured nonsmooth optimization} \label{sec_gen}
Our goal in this section is to show that the gradient sliding techniques developed
in Sections~\ref{sec_DGS} and \ref{sec_SGS}
can be further generalized to some other important classes of CP problems.
More specifically, we first study in Subsection~\ref{sec_strong} the composite CP problems in \eqnok{cp} with $f$ being strongly convex,
and then consider in Subsection~\ref{sec_saddle} the case where $f$ is a special nonsmooth function given in a bi-linear saddle point form.
Throughout this section, we assume that
the nonsmooth component $h$ is represented by a $\SO$ (see Section~1). 
It is clear that our discussion covers also the deterministic
composite problems as certain special cases by setting $\sigma = 0$ in
\eqnok{assum_FO_b} and \eqnok{assum_FO_c}.

\subsection{Strongly convex optimization} \label{sec_strong}
In this section, we assume that the smooth component $f$ in \eqnok{cp} is strongly convex, i.e.,
$\exists \mu > 0$ such that
\beq \label{strongcovex}
f(x) \ge f(y) + \langle \nabla f(y), x - y \rangle + \frac{\mu}{2} \|x-y\|^2, \ \ \forall x, y \in X.
\eeq
In addition, throughout this section, we assume that the prox-function grows quadratically so that
\eqnok{quad_grow} is satisfied. 

One way to solve these strongly convex composite problems is to apply the aforementioned accelerated
stochastic approximation algorithm which would require ${\cal O} (1/\epsilon)$
evaluations for $\nabla f$ and $h'$
to find an $\epsilon$-solution of \eqnok{cp}~\cite{GhaLan12-2a,GhaLan13-1}. However, we will show in this subsection that
this bound on the number of evaluations for $\nabla f$ can be significantly reduced
to ${\cal O}(\log (1/\epsilon))$, by properly restarting the
SGS algorithm in Section~\ref{sec_SGS}. This multi-phase stochastic gradient sliding (M-SGS) algorithm 
is formally described as follows.


\begin{algorithm} [H]
	\caption{The multi-phase stochastic gradient sliding (M-SGS) algorithm}
	\label{algGenericMSGS}
	\begin{algorithmic}
\State 
\noindent {\bf Input:} Initial point $y_0 \in X$, iteration limit $N_0$,  and an
initial estimate $\Delta_0$ s.t.
$
\Psi(y_0) - \Psi^* \le \Delta_0.
$

 
\For {$s=1, 2, \ldots, S$ }
\State 
Run the SGS algorithm with $x_0 = y_{s-1}$, $N = N_0$, 
$\{p_t\}$ and $\{\theta_t\}$ in \eqnok{spec_theta}, $\{\beta_k\}$ and $\{\gamma_k\}$ in \eqnok{spec_alpha_theta}, and
$\{T_k\}$ in \eqnok{def_TK_S} with $\tilde D = \Delta_0 / (\nu \mu 2^{s})$, and let 
$y_s$ be its output solution.
\EndFor
\State 
\noindent {\bf Output:} $y_S$.
	\end{algorithmic}
\end{algorithm}

We now establish the main convergence properties of the M-SGS algorithm described above.

\begin{theorem} \label{theorem_strong}
If $
N_0 = \left \lceil 4 \sqrt{2 L/(\nu \mu)} \right \rceil
$ in the MGS algorithm, 
then 
\beq \label{bnd_phase}
\bbe[\Psi(y_s) - \Psi^*] \le \frac{\Delta_0}{2^s}, \ \ s \ge 0.
\eeq
As a consequence, the total number of evaluations for $\nabla f$ and 
$H$, respectively, required by the M-SGS algorithm to find a
stochastic $\epsilon$-solution of \eqnok{cp}
can be bounded by 
\beq \label{bnd_grad_eva_strong}
{\cal O} \left( \sqrt{\frac{L}{\nu \mu}} \log_2 \max \left\{ \frac{\Delta_0}{\epsilon}, 1\right\}\right)
\eeq
and
\beq \label{bnd_grad_eva_strong1}
{\cal O} \left(\frac{M^2+\sigma^2}{\nu \mu \epsilon} + \sqrt{\frac{L}{\nu \mu}} \log_2 \max \left\{ \frac{\Delta_0}{\epsilon}, 1\right\} \right).
\eeq
\end{theorem}

\begin{proof}
We show \eqnok{bnd_phase} by induction. Note that \eqnok{bnd_phase} clearly holds for $s=0$
by our assumption on $\Delta_0$. Now assume that \eqnok{bnd_phase} holds at phase $s-1$, i.e., $\Psi(y_{s-1}) - \Psi^* \le \Delta_0 / 2^{(s-1)}$
for some $s \ge 1$. In view of Corollary~\ref{stoch_corol1} and the definition of $y_s$, we have
\begin{align*}
\bbe[\Psi(y_s) - \Psi^*| y_{s-1}] &\le \frac{2 L}{N_0 (N_0+1)} \left[\frac{3 V(y_{s-1}, x^*)}{\nu} + 4 \tilde D \right]\\
& \le \frac{2 L}{N_0^2} \left[\frac{6}{\nu \mu} (\Psi(y_{s-1}) - \Psi^*) + 4 \tilde D \right].
\end{align*}
where the second inequality follows from the strong convexity of $\Psi$ and \eqnok{quad_grow}.
Now taking expectation on both sides of the above inequality w.r.t. $y_{s-1}$, and using
the induction hypothesis and the definition of $\tilde D$ in
the M-SGS algorithm, we conclude that
\[
\bbe[\Psi(y_s) - \Psi^*] \le \frac{2L}{N_0^2} \frac{8 \Delta_0}{\nu \mu 2^{s-1}}  \le \frac{\Delta_0}{2^{s}},
\]
where the last inequality follows from 
the definition of $N_0$. Now, by \eqnok{bnd_phase}, the total number of phases
performed by the M-SGS algorithm can be bounded by
$S = \lceil \log_2 \max \left\{ \frac{\Delta_0}{\epsilon}, 1\right\} \rceil$. Using this observation,
we can easily see that 
the total number of gradient evaluations of $\nabla f$ is given by $N_0 S$, which is bounded by
\eqnok{bnd_grad_eva_strong}.  Now let us provide a bound on total number
of stochastic subgradient evaluations of $h'$.
Without loss of generality, let us assume that $\Delta_0 > \epsilon$.
Using the previous bound on $S$ and the definition of $T_k$,
the total number of stochastic subgradient evaluations of $h'$ can be bounded by
\begin{align*}
\sum_{s=1}^S \sum_{k=1}^{N_0} T_k &\le \sum_{s=1}^S \sum_{k=1}^{N_0} \left( \frac{\nu \mu N_0 (M^2+\sigma^2) k^2}{\Delta_0 L^2} 2^{s} +1\right) \\
&\le \sum_{s=1}^S \left[\frac{\nu \mu N_0 (M^2+\sigma^2)}{3 \Delta_0 L^2}(N_0+1)^3 2^{s}  + N_0\right]\\
& \le \frac{\nu \mu N_0 (N_0 + 1)^3 (M^2+\sigma^2)}{3 \Delta_0 L^2} 2^{S+1} + N_0 S\\
&\le \frac{4 \nu \mu N_0 (N_0 + 1)^3 (M^2+\sigma^2)}{3 \epsilon L^2} + N_0 S.
\end{align*}
This observation, in view of the definition of $N_0$, then clearly implies the bound in \eqnok{bnd_grad_eva_strong1}.
\end{proof}

\vgap

We now add a few remarks about the results obtained in
Theorem~\ref{theorem_strong}. Firstly, the M-SGS algorithm possesses optimal complexity bounds
in terms of the number of gradient evaluations for $\nabla f$ and subgradient evaluations for $h'$,
while existing algorithms only exhibit optimal complexity bounds on the number of
stochastic subgradient evaluations (see~\cite{GhaLan13-1}).
Secondly, in Theorem~\ref{theorem_strong}, we only establish the optimal convergence of the 
M-SGS algorithm in expectation. It is also possible to establish the optimal convergence of
this algorithm with high probability by making use of the light-tail assumption in \eqnok{assum_FO_c}
and a domain shrinking procedure similarly to the one studied in Section~3 of \cite{GhaLan13-1}.

\subsection{Structured nonsmooth problems} \label{sec_saddle}
Our goal in this subsection is to further generalize the gradient sliding algorithms
to the situation
when $f$ is nonsmooth, but can be closely approximated by a certain smooth convex function.

More specifically, we assume that
$f$ is given in the form of
\beq \label{saddleclass}
f(x) = \max_{y \in Y} \langle Ax, y \rangle - J(y),
\eeq
where $A:\bbr^n \to \bbr^m$ denotes a linear operator, $Y$ is a closed convex set,
and $J: Y \rightarrow \R $ is a relatively simple, proper,
convex, and lower semi-continuous (l.s.c.) function (i.e., problem \eqnok{sm-approx} below is easy to solve). 
Observe that if $J$ is the convex conjugate of some convex function $F$ and $Y \equiv {\cal Y}$,
then problem~\eqnok{cp} with $f$ given in \eqnok{saddleclass} can be written
equivalently as 
\[
\min_{x \in X} h(x) + F(Ax),
\] 
Similarly to the previous subsection, we focus on the situation when $h$ is represented by a $\SO$.
Stochastic composite problems in this form have wide applications in machine learning,
for example, to minimize
the regularized loss function of
\[
\min_{x \in X} \bbe_\xi[l(x, \xi)] + F(Ax),
\]
where $l(\cdot, \xi)$ is a convex loss function for any $\xi \in \Xi$
and $F(Kx)$ is a certain regularization (e.g.,  low rank tensor~\cite{KolBad09-1,TSHK11-1},
overlapped group lasso~\cite{JOV09,MaJeObBa11-1}, and graph regularization~\cite{JOV09,TSRZK05-1}).

Since $f$ in \eqnok{saddleclass} is nonsmooth, we cannot directly apply the
gradient sliding methods developed in the previous sections. However, as shown by Nesterov~\cite{Nest05-1}, the function $f(\cdot)$
in \eqnok{saddleclass} can be closely approximated by a class of smooth 
convex functions. More specifically,
for a given strongly convex function $v:Y \to \bbr$ such that
\beq \label{strong_v}
v(y) \ge v(x) + \langle \nabla v(x), y - x \rangle + \frac{\nu'}{2} \|y -x \|^2, \forall x, y \in Y
\eeq
for some $\nu' > 0$,
let us denote $c_v := \argmin_{y \in Y} v(y)$, $d(y) := v(y) - v(c_v) - \langle \nabla v(c_v), y - c_v \rangle$ and
\beq \label{def_cal_DX}
{\cal D}_{Y} := \max_{y \in Y} d(y).
\eeq 
Then the function $f(\cdot)$ in \eqnok{saddleclass} can be closely approximated by
\beq \label{sm-approx}
f_\eta(x) := \max\limits_{y}\left\{\langle A x, y \rangle - 
	J(y)-\eta \, d(y): \ y \in Y \right\}.
\eeq
Indeed, by definition we have $0 \le V(y) \le {\cal D}_{Y}$ and hence, for any
$\eta \ge 0$,
\beq \label{closeness1}
f(x) - \eta {\cal D}_Y \le f_\eta(x) \le f(x), \ \ \ \forall x \in X.
\eeq
Moreover, Nesterov~\cite{Nest05-1} shows that $f_{\eta}(\cdot)$ is differentiable and
its gradients are Lipschitz continuous with the Lipschitz constant given by 
\beq \label{new_ls}
{\cal L}_\eta := \frac{\|A\|^2}{\eta \nu'}.
\eeq

\vgap
We are now ready to present a smoothing stochastic gradient sliding (S-SGS)
method and study its convergence properties.

\begin{theorem} \label{the_SSGS}
Let $(\bx_k, x_k)$ be the search points generated by
a smoothing stochastic gradient sliding (S-SGS) method,
which is obtained by replacing $f$ with $f_\eta(\cdot)$ 
in the definition of $g_k$ in the SGS method.
Suppose that $\{p_t\}$ and $\{\theta_t\}$ in the $\SProxS$ procedure 
are set to \eqnok{spec_theta}. 
Also assume that $\{\beta_k\}$ and
$\{\gamma_k\}$ are set to \eqnok{spec_alpha_theta} and that $T_k$ is given by
\eqnok{def_TK_S}  with $\tilde D = 3 {\cal D}_X/(4 \nu)$ for some $N \ge 1$,
where ${\cal D}_X$ is given by \eqnok{def_DX}.
If  
\[
\eta = \frac{2 \|A\|}{N} \sqrt{\frac{3 {\cal D}_X}{\nu \nu' {\cal D}_{Y}} },
\]
then the total number of outer iterations and inner iterations performed by the
S-SGS algorithm to find an $\epsilon$-solution of \eqnok{cp} can be bounded by
\beq \label{S-SGS_bound1}
{\cal O} \left( \frac{\|A\|  \sqrt{{\cal D}_{X} {\cal D}_{Y} }}{\epsilon \sqrt{\nu \nu' } } \right)
\eeq
and \label{SGS_bound2}
\beq
{\cal O} \left\{ 
\frac{(M^2+\sigma^2) \|A\|^2 V(x_0, x^*)}{\nu \epsilon^2} +  \frac{\|A\|  \sqrt{{\cal D}_{Y} V(x_0,x^*)}}{\sqrt{\nu \nu' } \epsilon }
\right\},
\eeq
respectively.
\end{theorem}

\begin{proof}
Let us denote $\Psi_\eta(x) = f_\eta(x) + h(x) + \cX(x)$. In view of \eqnok{stoch_corol1_a} and \eqnok{new_ls}, we have
\begin{align*}
\bbe[\Psi_\eta(\bx_N) - \Psi_\eta(x)] &\le \frac{2 L_\eta }{N(N+1)}
\left[ \frac{3 V(x_0, x)}{\nu} + 4  \tilde D\right] \\ 
&= \frac{2 \|A\|^2 }{\eta \nu' N(N+1)}
\left[ \frac{3 V(x_0, x)}{\nu} + 4  \tilde D\right], \ \ \forall x \in X, \, N \ge 1.
\end{align*}
Moreover, it follows from \eqnok{closeness1} that
\[
\Psi_\eta(\bx_N) - \Psi_\eta(x) \ge \Psi(\bx_N) - \Psi(x) - \eta {\cal D}_{Y}.
\]
Combining the above two inequalities, we obtain
\[ 
\bbe[\Psi(\bx_N) - \Psi(x)] \le \frac{2 \|A\|^2 }{\eta \nu' N(N+1)}
\left[ \frac{3 V(x_0, x)}{\nu} + 4  \tilde D\right] + \eta {\cal D}_{Y}, \ \ \forall x \in X,
\]
which implies that
\beq \label{main_S-SGS}
\bbe[\Psi(\bx_N) - \Psi(x^*)] \le \frac{2 \|A\|^2 }{\eta \nu' N(N+1)}
\left[ \frac{3 {\cal D}_X}{\nu} + 4  \tilde D\right] + \eta {\cal D}_{Y}.
\eeq
Plugging the value of $\tilde D$ and $\eta$ into the above bound, 
we can easily see that 
\begin{align*}
\bbe[\Psi(\bx_N) - \Psi(x^*)]  &\le  \frac{4 \sqrt{3} \|A\| \sqrt{ {\cal D}_X {\cal D}_{Y} }}{\sqrt{\nu \nu' } N}, \ \ \forall x \in X, \, N \ge 1.
\end{align*}
It then follows from the above relation that the total number of outer iterations
to find an $\epsilon$-solution of problem \eqnok{saddleclass}
can be bounded by
\[
\bar N(\epsilon) = \frac{4 \sqrt{3} \|A\|  \sqrt{{\cal D}_{X} {\cal D}_{Y} }}{\sqrt{\nu \nu' } \epsilon}.
\]
Now observe that the total number of inner iterations is bounded by
\[
\sum_{k=1}^{\bar N(\epsilon)} T_k
= \sum_{k=1}^{\bar N(\epsilon)} \left[ \frac{(M^2+\sigma^2) \bar N(\epsilon) k^2}{\tilde D L_\eta^2} + 1\right]
= \sum_{k=1}^{\bar N(\epsilon)} \left[ \frac{(M^2+\sigma^2) \bar N(\epsilon) k^2}{\tilde D L_\eta^2} + 1 \right].
\]
Combining these two observations, we conclude that
the total number of inner iterations is bounded by \eqnok{SGS_bound2}.
\end{proof}

\vgap

In view of Theorem~\ref{the_SSGS}, by using the smoothing SGS algorithm, we can significantly reduce the 
number of outer iterations, and hence the number of times to access the linear operator $A$ and $A^T$,
from ${\cal O}(1/\epsilon^2)$ to ${\cal O}(1/\epsilon)$ in order to find an $\epsilon$-solution of \eqnok{cp},
while still maintaining the optimal bound on the total number of stochastic subgradient evaluations
for $h'$. It should be noted that, by using the result in Theorem~\ref{main_theorem_s}.b), we can show that
the aforementioned savings on the access to the linear operator $A$ and $A^T$ also hold with
overwhelming probability under the light-tail assumption in \eqnok{assum_FO_c} associated
with the $\SO$.  

\setcounter{equation}{0}
\section{Concluding remarks} \label{sec_remark}
In this paper, we present a new class of first-order method which can significantly reduce
the number of gradient evaluations for $\nabla f$ required to solve the composite problems in \eqnok{cp}. 
More specifically,
we show that by using these algorithms, the total number of gradient evaluations can be significantly reduced
from ${\cal O}(1/\epsilon^2)$ to ${\cal O}(1/\sqrt{\epsilon})$. As a result, these algorithms have the potential
to significantly accelerate first-order methods for solving
the composite problem in \eqnok{cp}, especially when the bottleneck exists in the computation  (or communication
in the case of distributed computing) of the gradient
of the smooth component, as happened in many applications. We also establish similar complexity bounds for solving an important class of stochastic 
composite optimization problems by developing the stochastic gradient sliding methods.
By properly restarting the gradient sliding algorithms,
we demonstrate that 
dramatic saving on gradient evaluations (from ${\cal O}(1/\epsilon)$ to ${\cal O}(\log (1/\epsilon)$) can be achieved 
for solving 
strongly convex problems.
 Generalization to the case when $f$ is nonsmooth but
possessing a bilinear saddle point structure has also been discussed.

It should be pointed out that this paper focuses only on theoretical studies for the convergence properties
associated with the gradient sliding algorithms. The practical performance for these algorithms, however, will certainly
depend on our estimation for a few problem parameters, e.g., the Lipschitz constants $L$ and $M$.
In addition, the sliding periods $\{T_k\}$ in both GS and SGS
have been specified in a conservative way to obtain the optimal complexity bounds for gradient and subgradient 
evaluations. We expect that the practical performance of these 
algorithms will be further improved with proper incorporation of certain adaptive search procedures on $L$, $M$, and
$\{T_k\}$, which will be very interesting research topics in the future.

\bibliographystyle{plain}
\bibliography{../glan-bib}

\end{document}